\documentclass{amsart}
\usepackage{amssymb,latexsym,array,graphicx}
\usepackage[ruled]{algorithm}
\usepackage{algorithmic}
\usepackage{mathptmx,courier}
\usepackage[scaled=0.95]{helvet}

\title{Bimonotone enumeration}

\author{Michael Eisermann}

\address{Institut Fourier, Universit\'e Grenoble I, France}

\email{Michael.Eisermann@ujf-grenoble.fr}

\urladdr{http://www-fourier.ujf-grenoble.fr/$\sim$eiserm}

\date{\today}

\newtheorem{theorem}{Theorem}%[section]
\newtheorem{lemma}[theorem]{Lemma}
\newtheorem{proposition}[theorem]{Proposition}
\newtheorem{corollary}[theorem]{Corollary}

\theoremstyle{definition}
\newtheorem{definition}[theorem]{Definition}
\newtheorem{example}[theorem]{Example}

\newtheorem{remark}[theorem]{Remark}

\numberwithin{equation}{section}
\newcommand{\sref}[1]{\textsection\ref{#1}}

\newcommand{\algoinput}{\textbf{Input:}}
\newcommand{\algorequire}{\textbf{Requires:}}
\newcommand{\algooutput}{\textbf{Output:}}
\newcommand{\algoensure}{\textbf{Ensures:}}
\newenvironment{specification}{%
  \newcommand{\INPUT}{\item[\mbox{\algoinput}\hfill]}%
  \newcommand{\REQUIRE}{\item[\mbox{\algorequire}\hfill]}%
  \newcommand{\OUTPUT}{\item[\mbox{\algooutput}\hfill]}%
  \newcommand{\ENSURE}{\item[\mbox{\algoensure}\hfill]}%
  \begin{list}{}{ \topsep0pt \partopsep0pt \parskip0pt %
      \itemindent0pt \listparindent0pt \labelwidth 4.5em \labelsep 0.5em %
      \leftmargin\labelwidth \addtolength{\leftmargin}{\labelsep} %
      } \itemsep0pt \parsep0pt }{\end{list}}

\newcommand{\algorithmicreturn}{\textbf{return}}
\newcommand{\RETURN}{\STATE\algorithmicreturn\ }
\newcommand{\IFTHEN}[2]%
{\STATE\algorithmicif\ #1\ \algorithmicthen\ #2\ \algorithmicendif}
\newcommand{\IFTHENELSE}[3]%
{\STATE\algorithmicif\ #1\ \algorithmicthen\ #2\ \algorithmicelse\ #3 \algorithmicendif}
\newcommand{\WHILEDO}[2]%
{\STATE\algorithmicwhile\ #1\ \algorithmicdo\ #2\ \algorithmicendwhile}
\newcommand{\DOWHILE}[2]%
{\STATE\algorithmicdo\ #1 \algorithmicwhile\ #2 \algorithmicendwhile}
\newcommand{\REPEATUNTIL}[2]%
{\STATE\algorithmicrepeat\ #1\ \algorithmicuntil\ #2}
\renewcommand{\algorithmiccomment}[1]{\hfill// \makebox[60mm][l]{\textit{#1}}}

\newcommand{\N}{\mathbb{N}}                     % natural numbers
\newcommand{\Z}{\mathbb{Z}}                     % integer numbers
\newcommand{\Q}{\mathbb{Q}}                     % integer numbers
\newcommand{\R}{\mathbb{R}}                     % real numbers
\newcommand{\minus}{\smallsetminus}             % setminus
\newcommand{\im}{\operatorname{Im}}             % image of a map
\newcommand{\pr}{\operatorname{pr}}             % projection map
\newcommand{\tle}{\leqslant}                    % total order: less or equal
\newcommand{\tge}{\geqslant}                    % total order: greater or equal
\newcommand{\sle}{\leq}                         % semimonotone order 
\newcommand{\ble}{\mathbin{\leq\!\!\!\leq}}     % bimonotone order
\newcommand{\amin}{a_\mathrm{min}}              % minimal element
\newcommand{\amax}{a_\mathrm{max}}              % maximal element
\newcommand{\bmin}{b_\mathrm{min}}              % minimal element
\newcommand{\bmax}{b_\mathrm{max}}              % maximal element
\newcommand{\pred}[1]{\pi{#1}}                  % predecessor function
\renewcommand{\succ}[1]{\sigma{#1}}             % successor function
\newcommand{\supper}{{\smash{+}}}               % semimonotone upper set
\newcommand{\bupper}{{\smash{\#}}}              % bimonotone upper set
\newcommand{\Min}{\operatorname{Min}}           % set of minimal elements

\newcommand{\List}[1]{\left\langle\; #1 \;\right\rangle}
\newcommand{\Queue}[1]{\left\langle\; #1 \;\right\rangle}
\newcommand{\taxicab}[1]{\operatorname{taxicab}(#1)}

\begin{document} %%%%%%%%%%%%%%%%%%%%%%%%%%%%%%%%%%%%%%%%%%%%%%%%%%%%%%%%%%%%
%%%%%%%%%%%%%%%%%%%%%%%%%%%%%%%%%%%%%%%%%%%%%%%%%%%%%%%%%%%%%%%%%%%%%%%%%%%%%

\begin{abstract}
  Solutions of a diophantine equation $f(a,b) = g(c,d)$, 
  with $a,b,c,d$ in some finite range, can be efficiently 
  enumerated by sorting the values of $f$ and $g$ 
  in ascending order and searching for collisions.
  This article considers functions $f \colon \N \times \N \to \Z$ 
  that are bimonotone in the sense that $f(a,b) \le f(a',b')$ 
  whenever $a \le a'$ and $b \le b'$.  A two-variable polynomial 
  with non-negative coefficients is a typical example. 
  The problem is to efficiently enumerate all pairs $(a,b)$ 
  such that the values $f(a,b)$ appear in increasing order.
  We present an algorithm that is memory-efficient and highly parallelizable.
  In order to enumerate the first $n$ values of $f$, the algorithm 
  only builds up a priority queue of length at most $\sqrt{2n}+1$.
  In terms of bit-complexity this ensures that the algorithm 
  takes time $O( n \log^2 n )$ and requires memory $O( \sqrt{n} \log n )$, 
  which considerably improves on the memory bound $\Theta( n \log n )$ provided by 
  a na\"ive approach, and extends the semimonotone enumeration 
  algorithm previously considered by R.L.\,Ekl and D.J.\,Bernstein.
\end{abstract}

 % 11D25, 11D41

%68 (1940-now) Computer science
%68P (1985-now) Theory of data 
%68P05 (1985-now) Data structures
%68P10 (1985-now) Searching and sorting
%68P15 (1985-now) Database theory
%68P20 (1985-now) Information storage and retrieval
%68P25 (1985-now) Data encryption [See also 94A60, 81P68]
%68P30 (2000-now) Coding and information theory

%68W (2000-now) Algorithms
%68W05 (2000-now) Nonnumerical algorithms
%68W10 (2000-now) Parallel algorithms
%68W15 (2000-now) Distributed algorithms
%68W20 (2000-now) Randomized algorithms
%68W25 (2000-now) Approximation algorithms
%68W40 (2000-now) Analysis of algorithms [See also 68Q25]

%68E (1980-1984) Discrete mathematics 
%68Exx (1980-1984) Discrete mathematics
%68E05 (1980-1984) Sorting, searching
 
%68A (1973-1990) Computer science 
%68Axx (1973-1990) Computer science
%68A10 (1973-1979) Algorithms [See also 02E10.]
%68A20 (1973-1979) Computational complexity and efficiency

%11Y (1985-now) Computational number theory [See also 11-04] 
%11Yxx (1985-now) Computational number theory [See also 11-04]
%11Y16 (1985-now) Algorithms; complexity [See also 68Q25]
%11Y50 (1985-now) Computer solution of Diophantine equations

%11D (1985-now) Diophantine equations [See also 11Gxx, 14Gxx] 
%11Dxx (1985-now) Diophantine equations [See also 11Gxx, 14Gxx]
%11D25 (1985-now) Cubic and quartic equations
%11D41 (1985-now) Higher degree equations; Fermat's equation
%11D45 (2000-now) Counting solutions of Diophantine equations

%11P (1985-now) Additive number theory; partitions 
%11Pxx (1985-now) Additive number theory; partitions
%11P05 (1985-now) Waring's problem and variants
 

%% \copyrightinfo{2003-2008}{Michael Eisermann}

%%%%%%%%%%%%%%%%%%%%%%%%%%%%%%%%%%%%%%%%%%%%%%%%%%%%%%%%%%%%%%%%%%%%%%%%%%%%%

% \headline{20mm}{\eprintinfo}

\vspace*{20mm}

\maketitle

%%%%%%%%%%%%%%%%%%%%%%%%%%%%%%%%%%%%%%%%%%%%%%%%%%%%%%%%%%%%%%%%%%%%%%%%%%%%%

\section{Introduction and statement of results} \label{sec:Introduction}

\subsection{Motivation} \label{sub:Motivation}

Given polynomial functions $f,g \colon \N\times\N \to \Z$, 
how can we efficiently enumerate solutions of the equation 
$f(a,b) = g(c,d)$? One standard way to do this is to sort 
the sets $F = \{\, (f(a,b),a,b) \mid 1\le a,b\le N \,\}$
and $G = \{\, (g(c,d),c,d) \mid 1\le c,d\le N \,\}$ into ascending order 
with respect to the first coordinate and to look for collisions.
As stated, this requires to store all elements before sorting,
which consumes memory $\Theta(n \log n)$,
where $n = N^2$ is the number of values to enumerate,
and time between $\Omega(n \log n)$ and $O(n \log^2 n)$.

The present article develops a less memory consuming algorithm 
under the hypothesis that $f$ and $g$ are \emph{bimonotone}, 
that is, monotone in each variable.
This is sufficiently often the case to be interesting, for example 
when $f$ and $g$ are given by polynomials with non-negative coefficients.
Given a bimonotone function $f$, Algorithm \ref{algo:BimonotoneEnumeration}, 
discussed below, produces a stream $x_1,x_2,x_3,\dots$ enumerating 
all parameters $x_i = (a_i,b_i)$ in the domain of $f$ 
such that $f(x_1) \tle f(x_2) \tle f(x_3) \tle \dots$. 
Having at hand such sorted enumerations for $f$ and $g$,
one can easily enumerate solutions of the equation $f(x)=g(y)$:
start with $i=1$ and $j=1$;
whenever $f(x_i)<g(y_j)$, increment $i$;
whenever $f(x_i)>g(y_j)$, increment $j$.
If eventually $f(x_i)=g(y_j)$, then output 
the solution $(x_i,y_j)$ and continue searching.

\subsection{Main result} \label{sub:MainResult}

The idea of sorted enumeration has been applied by D.J.\,Bernstein \cite{Bernstein:2001}
with great success to equations of the special form $p(a) + q(b) = r(c) + s(d)$.
We generalize his approach to arbitrary bimonotone functions.
The main result can be stated as follows:

\begin{theorem}
  Suppose that $f \colon \N \times \N \to \Z$ is bimonotone 
  and proper in the sense that for every $z \in \Z$ 
  only finitely many pairs $(a,b)$ satisfy $f(a,b)\le z$.
  Then Algorithm \ref{algo:BimonotoneEnumeration} stated below 
  produces a stream enumerating all pairs $(a,b) \in \N \times \N$ 
  such that the values $f(a,b)$ appear in increasing order.
  While enumerating the first $n$ values, the algorithm 
  only builds up a priority queue of length $m \le \sqrt{2n}+1$.
  If $f$ is a polynomial, this ensures that the algorithm 
  takes time $O( n \log^2 n )$ and requires memory $O( \sqrt{n} \log n )$.
\end{theorem}

The precise bound $m \le \sqrt{2n}+1$ is free of hidden constants 
and thus uniformly valid for all bimonotone functions $f$.
The less explicit bounds of time $O( n \log^2 n )$ and 
memory $O( \sqrt{n} \log n )$ concern the bit-complexity  
and the hidden constants necessarily depend on $f$.
We shall assume throughout that $f$ behaves polynomially, 
see \sref{sub:TimeAndMemory}.

To place this result into perspective, notice that 
the time requirement $O(n\log^2 n)$ is nearly optimal:
enumerating $n$ elements obviously needs $n$ iterations, 
and one $\log n$ factor is due to their increasing size.
On the other hand, the standard approach would require 
memory $\Theta( n \log n )$ to store \emph{all} values 
before outputting them.  Here the stream approach can achieve 
considerable savings and reduce memory to $O( \sqrt{n} \log n )$.

\begin{example}
  Consider $f(a,b) = p(a) + q(b)$ where $p$ and $q$ 
  are non-decreasing polynomial functions of degree 
  $\alpha = \deg p$ and $\beta = \deg q$, respectively.
  Assuming $1 \le \alpha \le \beta$, Algorithm \ref{algo:BimonotoneEnumeration} 
  builds up a priority queue of length $m \in \Theta( n^\varepsilon )$ with 
  $\varepsilon = \frac{\alpha}{\alpha+\beta} \in [0,\frac{1}{2}]$.
\end{example}

This illustrates that in the uniform bound $m \in O( n^{1/2} )$,
stated in the theorem for all bimonotone functions,  
the exponent $\frac{1}{2}$ cannot be improved.  
Notwithstanding, the algorithm performs better
on certain subclasses of bimonotone functions,
where $\varepsilon < \frac{1}{2}$.

\begin{remark}
  The predecessor of our algorithm is semimonotone enumeration, 
  recalled in \sref{sec:SemimonotoneEnumeration}.
  It was devised in \cite{Ekl:1996,Bernstein:2001} 
  for polynomials of the form $f(a,b) = p(a) + q(b)$,
  where it provides the desired memory bound $O( \sqrt{n} \log n )$.
  In the more general setting of bimonotone functions, however, 
  we show that it only guarantees the memory bound $O( n \log n )$
  and the exponent $1$ can in general not be improved.
  See \sref{sec:AsymptoticComplexity} for a detailed discussion.
\end{remark}

As an additional benefit, our algorithm turns out to be highly parallelizable:

\begin{remark}
  Algorithm \ref{algo:BimonotoneEnumeration} can be adapted 
  to enumerate only those pairs $(a,b) \in \N\times\N$ for which 
  the values $f(a,b)$ lie in a given interval $[z_1,z_2]$.  
  Time and memory requirements are essentially the same as before;
  only initialization induces some additional cost 
  and can usually be neglected.  This means that searching
  solutions $f(a,b)=g(c,d)$ can be split up into disjoint 
  intervals and thus parallelized on independent machines
  (see \sref{sec:Parallelization}).
\end{remark}

\subsection{How this article is organized} \label{sub:Overview}

Section \ref{sec:SortedEnumeration} introduces the necessary notation
and recalls the generic algorithm of sorted enumeration for 
an arbitrary map $f\colon X\to Z$, where $X$ is a finite set.
Section \ref{sec:SemimonotoneEnumeration} discusses a refined algorithm,
essentially due to R.L.\,Ekl \cite{Ekl:1996} and D.J.\,Bernstein \cite{Bernstein:2001},
under the hypothesis that $f\colon A\times B\to Z$ is semimonotone, that is, 
monotone in the first variable. Section \ref{sec:BimonotoneEnumeration}
develops a sorted enumeration algorithm for bimonotone functions, 
and Section \ref{sec:AsymptoticComplexity} analyses the asymptotic complexity.
Section \ref{sec:Parallelization} highlights 
the intrinsically parallel structure of such a search problem.
Section \ref{sec:BimonotoneDomains} generalizes our algorithms
to functions $f \colon X \to Z$ restricted to suitable domains
$X \subset A \times B$ that are of of practical interest.
Finally, Section \ref{sec:Applications} briefly indicates
applications to diophantine enumeration problems,
such as the taxicab problem.

% We emphasize that the principle of sorted enumeration relies solely on the order,
% and no further structure is taken into account for the algorithms in 
% \sref{sec:SortedEnumeration}--\sref{sec:Parallelization}.
% The ring structure enters the picture in the examples and applications to 
% diophantine equations in \sref{sec:Applications}.

\subsection{Acknowledgements}

I thank the anonymous referee for his thorough critique, harsh but fair,
which substantially contributed to improve the exposition.

% I thank the anonymous referee for his thorough critique, harsh but fair,
% of the first version of this article.  It substantially contributed to improve 
% the exposition, and in particular led me to include the asymptotic analysis 
% now found in \sref{sec:AsymptoticComplexity}.

%%%%%%%%%%%%%%%%%%%%%%%%%%%%%%%%%%%%%%%%%%%%%%%%%%%%%%%%%%%%%%%%%%%%%%%%%%%%%

\section{Sorted enumeration for arbitrary functions} \label{sec:SortedEnumeration}

Before discussing more sophisticated versions, let us first describe 
the general problem of sorted enumeration and recall its generic solution.

\subsection{The generic problem} \label{sub:GenericProblem}

Throughout this article we consider an ordered set $(Z,\tle)$.
By \emph{order} we always mean a reflexive, transitive relation
that is complete and antisymmetric, i.e.\ each pair $z \ne z'$ 
in $Z$ satisfies either $z \tle z'$ or $z' \tle z$.
Without completeness we may have neither $z \tle z'$ nor $z' \tle z$,
in which case we speak of a \emph{partial} order.
Without antisymmetry we may have both $z \tle z'$ and $z' \tle z$,
in which case we speak of a \emph{preorder}.

We assume that $X$ is a finite or countably infinite set.
An \emph{enumeration} of $X$ is a stream $x_1,x_2,x_3,\dots$ 
in which each element of $X$ occurs exactly once.
Such an enumeration is \emph{monotone} or \emph{sorted} 
with respect to $f \colon X \to Z$ if 
it satisfies $f(x_1) \tle f(x_2) \tle f(x_3) \tle \dots$.
Whenever the function $f$ is understood from the context, 
we will simply speak of a \emph{sorted enumeration} of $X$.

\begin{remark}
  The map $f \colon X \to Z$ can be used to pull back the order $\tle$ 
  from $Z$ to the initially unordered set $X$.  More explicitly, 
  we define $x \preccurlyeq x'$ if and only if $f(x) \tle f(x')$.
  A sorted enumeration of $X$ is thus a stream in which
  all elements of $X$ appear in increasing order 
  with respect to the preorder $\preccurlyeq$.
\end{remark}

\subsection{The generic algorithm} \label{sub:GenericAlgorithm}

In the general setting, where $X$ is finite and $f$ has no further structure,
there is essentially only one way to produce a sorted enumeration:

\begin{algorithm}[H]
  \caption{\quad Sorted enumeration for an arbitrary function}
  \label{algo:SortedEnumeration}
  
  \begin{specification}
    \REQUIRE  a function $f\colon X\to Z$ from a finite set $X$ to an ordered set $Z$
    \OUTPUT   an enumeration of $X$, monotone with respect to $f$
  \end{specification}
  
  \smallskip\hrule
  
  \begin{algorithmic}[1]
    \STATE Generate a list $L$ of all pairs $(f(x),x)$ with $x\in X$.
    \STATE Sort the list $L$ according to the first coordinate $f(x)$.
    \STATE Output the arguments $x$ as sorted in the list $L$.
  \end{algorithmic}
\end{algorithm}

Algorithm \ref{algo:SortedEnumeration} is obviously correct.
Given a set $X$ of size $n$, generating and reading the list $L$ % of size $n$
takes $n$ iterations, while sorting requires $O( n \log n )$ operations.
Not much optimization can be expected concerning these time requirements,
since enumeration (sorted or not) takes at least $n$ iterations. 
Memory requirements, however, may be far from optimal,
and the more specialized algorithms discussed below 
will mainly be concerned with minimizing the use of temporary memory.

\subsection{Time and memory requirements} \label{sub:TimeAndMemory}

Throughout this article we use standard asymptotic notation,
as in \cite[\textsection9]{GrahamKnuthPatashnik:1989}.
% For later reference let us first have a closer look 
% at the cost of Algorithm \ref{algo:SortedEnumeration}.
It is customary to consider the cost for storing 
and handling elements $x$ and $f(x)$ to be constant. 
This is no longer realistic when the size $n = |X|$ grows without bound.
As a typical example, consider a polynomial function 
$f \colon \N \to \Z$ restricted to $X=\{1,\dots,n\}$. 
If each element $x \in X$ is stored in binary form,
the maximal memory required is $\Theta(\log n)$.
Likewise, the maximal time to calculate, copy, 
and compare values $f(x)$ is $\Theta(\log n)$,
neglecting factors of order $\log\log n$ or less.
Most elements require nearly maximum cost, 
so we shall only consider the worst case.

In general, we say that $f$ \emph{behaves polynomially} 
if the bit-complexity per element is $\Theta(n)$, as above.
In this case we arrive at the following more realistic account:

\begin{proposition} \label{prop:NaiveCost}
  In order to enumerate a set $X$ of size $n$, 
  the generic Algorithm \ref{algo:SortedEnumeration} 
  builds up a list of size $m=n$, and thus requires 
  time $O(n \log^2 n)$ and memory of size $\Theta(n \log n)$.
  \qed
\end{proposition}

%%%%%%%%%%%%%%%%%%%%%%%%%%%%%%%%%%%%%%%%%%%%%%%%%%%%%%%%%%%%%%%%%%%%%%%%%%%%%

\section{Sorted enumeration for semimonotone functions} \label{sec:SemimonotoneEnumeration}

In this section we consider a semimonotone function $f\colon A\times B\to Z$.
By this we mean that $(A,\tle)$ is an ordered set
and $a\tle a'$ implies $f(a,b) \tle f(a',b)$ for all $b\in B$.
This is the same as saying that $f$ is monotone
with respect to the partial order $(a,b) \sle (a',b')$ 
defined by the condition $a \tle a'$ and $b = b'$.

\subsection{The idea of semimonotone enumeration} \label{sub:SemimonotoneIdea}

We will first assume that $A$ and $B$ are finite sets.
This entails that $(A,\tle)$ is isotonic 
to an interval $\{1,\dots,l\}$ of integers.
The minimal and maximal element of $A$ is 
denoted by $\amin$ and $\amax$, respectively,
and the successor function is denoted by $a\mapsto\succ{a}$.
Of course $\amax$ cannot have a successor in $A$,
so by convention we set $\succ{\amax}=+\infty$.

We equip $X = A \times B$ with the partial order $\sle$ as defined above.
Given a subset $X_i \subset X$, we denote by $M_i = \Min(X_i)$ 
the set of its minimal elements. Conversely, $M_i$ defines its upper set
$M_i^\supper = \{\, x\in X \mid m\sle x \text{ for some } m\in M_i \,\}$.
Figure \ref{fig:minima1} shows a subset $X_i$ (indicated by crosses)
together with its set of minima $M_i$ (circled crosses). 
In this example $X_i$ is \emph{saturated} in the sense that $X_i = M_i^\supper$.

\begin{figure}[htbp]
  \includegraphics[height=50mm]{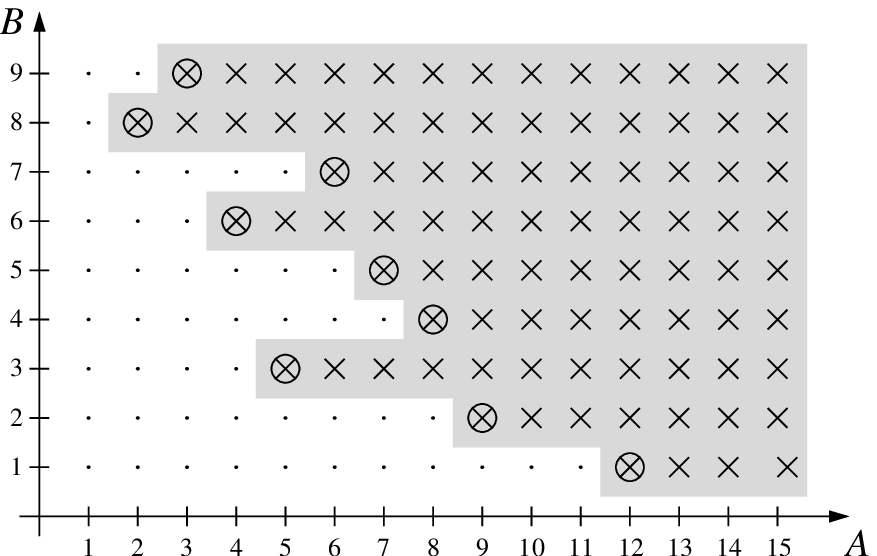}
  \caption{A subset of $A\times B$ and its minima with respect to $\sle$}
  \label{fig:minima1}
\end{figure} 

Since $f$ is monotone with respect to $\sle$,
the minimum of $f(X_i)$ is attained on $M_i$.
It thus suffices to find $x_i\in M_i$ realizing $f(x_i) = \min f(M_i)$.
We can then output $x_i$ and continue with the set $X_{i+1} = X_i \minus \{x_i\}$.
Notice that $X_{i+1}$ is again saturated and 
$M_{i+1}$ can be easily constructed from $M_i$.
This is the idea of Algorithm \ref{algo:SemimonotoneEnumeration} below.

%% It suffices to replace $x_i=(a,b)$ by its successor $x_i'=(\succ{a},b)$. 
%% Thus, instead of the entire set $X_i$ we only need 
%% to keep track of $M_i$, the set of minimal elements.
%% In conclusion we see that instead of searching $X_i$
%% we only need to keep track of the set $M_i$ of minimal elements.
%% which has the advantage that $M_i$ is in general much smaller than $X_i$.

\subsection{Suitable data structures} \label{sub:SemimonotoneDataStructures}

The following algorithm has been independently developed by 
R.L.\,Ekl \cite{Ekl:1996} and D.J.\,Bernstein \cite{Bernstein:2001},
and formalizes the above approach:
instead of handling the entire set $X_i$, 
it operates on two smaller sets, $M = \Min(X_i)$ and $F = f(M)$.
In order to efficiently find $x_i\in M$ realizing $f(x_i) = \min f(M)$,
we store the set of images $f(M)$ in a priority queue $F$. 
Recall that a priority queue $F$ for elements of $(Z,\tle)$ 
provides the following elementary operations:
\begin{itemize}
\item Inserting an element $z\in Z$ into $F$ (``push'').
\item Reading and removing a minimal element of $F$ (``pop'').
\end{itemize}

Priority queues are typically implemented using a heap or a binary tree; 
in either case the elementary operations need $O(\log m)$ steps,
where $m$ is the number of elements in the priority queue.
For a general presentation see Knuth \cite[\textsection5.2.3]{Knuth:vol3}.

\subsection{The semimonotone enumeration algorithm} \label{sub:SemimonotoneAlgorithm}

Instead of $f \colon A \times B \to Z$ it is more convenient 
to work with the map $f^* \colon A \times B \to Z \times A \times B$ 
defined by $f^*(a,b) = (f(a,b),a,b)$.  In our formulation of 
Algorithm \ref{algo:SemimonotoneEnumeration} we thus use 
a priority queue $F$ for elements in $Z \times A \times B$, 
sorted by the first coordinate. 

\begin{algorithm}[htbp]
  \caption{\quad Sorted enumeration for a semimonotone function} 
  \label{algo:SemimonotoneEnumeration}
  
  \begin{specification}
    \REQUIRE  a semimonotone function $f \colon A \times B \to Z$ % where $B$ is finite
    \OUTPUT   an enumeration of $A \times B$, monotone with respect to $f$
  \end{specification}
  
  \smallskip\hrule
  
  \begin{algorithmic}[1]
    \STATE Start with an empty priority queue $F$, then insert $f^*(\amin,b)$ for all $b\in B$.
    \WHILE{$F$ is non-empty}
      \STATE Remove a minimal element $f^*(a,b)$ from $F$ and output $(a,b)$.
      \IFTHEN{ $\succ{a} < +\infty$ }{ insert $f^*(\succ{a},b)$ into $F$ }
    \ENDWHILE
  \end{algorithmic}
\end{algorithm}

\begin{remark}
  All algorithms presented here can be regarded as templates, 
  to be instantiated for the given map $f \colon A \times B \to Z$.
  Alternatively, one could consider them as taking the sets 
  $A$ and $B$ and the map $f$ as input data.  In this case, of course, 
  we do not pass the entire sets $A$ and $B$ as parameters, nor the map $f$, 
  say as some subset of $A \times B \times Z$:  for finite sets
  this would be as inefficient as Algorithm \ref{algo:SortedEnumeration};
  for infinite sets it is simply impossible.

  Instead, it suffices to call some function that calculates 
  $f(a,b)$ for any given pair of parameters $a \in A$ and $b \in B$.
  To represent the sets $A$ and $B$, all we need is 
  the usual \emph{iterator} concept, providing a pointer to 
  the first (and possibly the last) element of the set and 
  a method for incrementing, denoted by $a \mapsto \succ{a}$ above. 
  (Algorithms \ref{algo:ConstructMinima} and \ref{algo:BimonotoneMinima} 
  also decrement, denoted by $b \mapsto \pred{b}$.)

  One can easily add suitable specifications when passing to concrete 
  implementations.  For the present general exposition, however, 
  we shall maintain the slightly coarser description, trying to strike 
  a balance between general concepts and implementation details.
\end{remark}

Algorithm \ref{algo:SemimonotoneEnumeration} is obviously correct.
The point is, as motivated above, that it usually uses less memory 
than the generic Algorithm \ref{algo:SortedEnumeration}.

\begin{proposition} \label{prop:SemimonotoneFinite}
  In order to enumerate a set $X = A \times B$ of size $n$,
  Algorithm \ref{algo:SemimonotoneEnumeration} 
  only builds up a priority queue of size $m = |B|$. 
  Let $f \colon \N \times \N \to \Z$ be a semimonotone map 
  that behaves polynomially.  Applied to subsets $A = \{1,\dots,l\}$
  and $B = \{1,\dots,m\}$ with $m \ge 2$, the algorithm thus takes 
  time $O(n \log n \log m)$ and requires memory $\Theta(m \log n)$.
\end{proposition}

% Applied to a semimonotone map $f \colon \N \times \N \to \Z$
% that behaves polynomially, the algorithm thus takes time 
% $O(n \log n \log m)$ and requires memory $\Theta(m \log n)$
% to enumerate $A \times B$ with $A = \{1,\dots,l\}$
% and $B = \{1,\dots,m\}$, $m \ge 2$.

\begin{proof}
  The algorithm needs memory to hold 
  $m$ elements $f^*(a,b)$ in the priority queue $F$.
  Since most elements need memory of size $\Theta(\log n)$, 
  we arrive at a total memory cost of $\Theta(m \log n)$.
  During each one of the $n$ iterations, the most time consuming operation 
  is updating the priority queue $F$ which requires time $O(\log n \log m)$.
  Here $m$ is the size of the queue and $\log n$ is the typical size of its elements.
\end{proof}

\begin{remark}
  Notice that in the degenerate case $|B|=1$, 
  Algorithm \ref{algo:SemimonotoneEnumeration} 
  simply enumerates $A$ in increasing order,
  which takes time $O(n \log n)$ and memory $\Theta(\log n)$.
  In the opposite extreme $|A|=1$, it sorts $B$ with respect to $f$ 
  via heap-sort. We thus fall back on the generic 
  Algorithm \ref{algo:SortedEnumeration},
  which takes time $O(m \log^2 m)$ and space $\Theta(m \log m)$.
  %% (cf.\ Proposition \ref{prop:NaiveCost}).
\end{remark}

\subsection{Enumerating infinite sets} \label{sub:SemimonotoneInfinite}

Sorted enumeration can be generalized from finite to infinite sets.
First of all, in order to be amenable to enumeration, 
$A$ must be either finite or isotonic to the natural numbers.
Moreover, we have to require that $f \colon A \times B \to Z$ be 
a \emph{proper} map in the sense that for every $z \in \im(f)$ 
only finitely many pairs $(a,b)$ satisfy $f(a,b) \le z$. 
(This condition actually implies that $A$ is finite or isotonic to $\N$.) 
Of course, we also have to assume that comparisons and all other operations 
are computable; as before their cost will be assumed to be of order $O(\log n)$.

\begin{proposition} \label{prop:SemimonotoneSemifinite}
  Suppose that $A$ is an infinite ordered set and $B$ is a finite set of size $m$.
  Let $f\colon A\times B\to Z$ be a proper semimonotone map.
  Then Algorithm \ref{algo:SemimonotoneEnumeration} produces a stream
  enumerating all pairs $(a,b)\in A\times B$ such that 
  the values $f(a,b)$ appear in increasing order.
  %% If $f$ behaves polynomially, then 
  Producing the first $n$ values takes time 
  $O(n \log n \log m)$ and requires memory $\Theta(m \log n)$.
\end{proposition}

\begin{proof}
  For every $z\in\im(f)$ the set $\{\, (a,b)\in A\times B \mid f(a,b)\le z \,\}$
  is finite, and thus contained in some finite product $[\amin,a_1] \times B$.
  Hence Algorithm \ref{algo:BimonotoneEnumeration} correctly enumerates 
  all parameters $(a,b)$ with $f(a,b)\le z$ as in the finite case.
  Since this is true for all $z$, the enumeration exhausts $A\times B$.
  Bit-complexity behaves as in Proposition \ref{prop:SemimonotoneFinite}.
\end{proof}

We wish to adapt semimonotone enumeration 
to the case where both $A$ and $B$ are infinite.
Algorithm \ref{algo:SemimonotoneEnumeration} is certainly
not suited for this task, because the initialization
will get stuck in an infinite loop.  As a necessary restriction
we require that $f \colon A \times B \to Z$ be proper, and
as before we assume that $f$ is monotone with respect to $A$.
For every $z \in \im(f)$, we can thus enumerate the finite set 
$\{ f \tle z \} := \{ (a,b) \in A \times B \mid f(a,b) \tle z \}$ 
by applying Algorithm \ref{algo:SemimonotoneEnumeration} to the relevant finite set 
$B(z) = \pr_2 \{ f \tle z \} = \{ b \in B \mid f(\amin,b) \le z \}$.
%% instead of the infinite set $B$.

In order to formulate an explicit algorithm, 
we assume that the set $B$ is ordered and that 
$f_2 \colon B \to Z$, $f_2(b) = f(\amin,b)$ is non-decreasing.
%% In this case $f$ could be called a \emph{sesquimonotone} function.
This is strictly weaker than demanding $f$ to be bimonotone, because 
we require monotonicity in $b$ only on the axis $\{\amin\} \times B$.
This technical condition ensures that we can 
easily construct the relevant finite set $B(z)$.
In fact, the monotonicity of $f_2$ is not at all restrictive, 
because we can \emph{choose} the order on $B$, for example by pulling back 
the order on $Z$ via $f_2$ to a preorder on $B$, and then refining 
to an order by arbitrating collisions. % disambiguating wherever necessary.
In other words, the order on $B$ is just a convenient way to encode
some preparatory analysis of the proper map $f_2 \colon B \to Z$.

This idea is formalized in Algorithm \ref{algo:SesquimonotoneEnumeration},
which is a slight variation of Algorithm \ref{algo:SemimonotoneEnumeration}.
The only difference is that it automatically adapts the relevant interval 
$B(z) = \left[ \bmin, \bmax \right[$ according to the level $z$ 
attained during the enumeration.  

\begin{algorithm}[htbp]
  \caption{\quad Sorted enumeration for a semimonotone function} 
  \label{algo:SesquimonotoneEnumeration}
  
  \begin{specification}
    \REQUIRE  a proper semimonotone function $f \colon A \times B \to Z$
    \OUTPUT   an enumeration of $A \times B$, monotone with respect to $f$.
  \end{specification}
  
  \smallskip\hrule
  
  \begin{algorithmic}[1]
    \STATE Initialize the priority queue $F \gets \Queue{ f^*(\amin,\bmin) }$ 
           and set $\bmax \gets \bmin$.
    \WHILE{$F$ is non-empty}
      \STATE Remove a minimal element $f^*(a,b)$ from $F$ and output $(a,b)$.
      \IFTHEN{ $\succ{a} < +\infty$ }{ insert $f^*(\succ{a},b)$ into $F$ }
      \IF{ $b = \bmax $ }
      \STATE Set $\bmax \gets \sigma\bmax$.
      \IFTHEN{ $\bmax < +\infty$ }{ insert $f^*(\amin,\bmax)$ into $F$ }
      \ENDIF
    \ENDWHILE
  \end{algorithmic}
\end{algorithm}

Here we have formulated Algorithm \ref{algo:SesquimonotoneEnumeration}
so that it applies to finite and infinite sets alike.  If $A$ or $B$
is infinite, then $\succ{a} < +\infty$ or $\bmax < +\infty$, respectively, 
is always true and the corresponding test can be omitted.

\begin{proposition} \label{prop:SemimonotoneInfinite}
  Suppose that both $A$ and $B$ are infinite ordered sets and that 
  $f \colon A \times B \to Z$ is a proper semimonotone function.
  We also assume that $b \mapsto f(\amin,b)$ is non-decreasing.
  Then Algorithm \ref{algo:SemimonotoneEnumeration} provides a sorted 
  enumeration of $A \times B$.  While enumerating the first $n$ values, 
  it builds up a priority queue of length $m \le n+1$.  This ensures that 
  the algorithm takes time $O( n \log^2 n )$ and memory $O( n \log n )$.
  \qed
\end{proposition}

Semimonotone functions are tailor-made for applications 
where we have monotonicity in $a$ but not necessarily in $b$.
They are halfway towards bimonotone functions,
which are more restrictive but support much better algorithms.
These will be discussed next.

%%%%%%%%%%%%%%%%%%%%%%%%%%%%%%%%%%%%%%%%%%%%%%%%%%%%%%%%%%%%%%%%%%%%%%%%%%%%%

\section{Sorted enumeration for bimonotone functions} \label{sec:BimonotoneEnumeration}

In this section we finally turn to bimonotone functions $f\colon A\times B\to Z$.
By this we mean that both $(A,\tle)$ and $(B,\tle)$ are ordered sets,
and that $a\le a'$ and $b\le b'$ implies $f(a,b)\le f(a',b')$.
This is the same as saying that $f$ is monotone 
with respect to the partial order $(a,b) \ble (a',b')$ 
defined by $a\tle a'$ and $b\tle b'$.

\subsection{The idea of bimonotone enumeration} \label{sub:BimonotoneIdea}

We will first assume that both sets $A$ and $B$ are finite.
Given a subset $X_i \subset X$ we denote by $M_i = \Min(X_i)$ 
the set of its minimal elements with respect to $\ble$.
Conversely, $M_i$ defines its upper set 
$M_i^\bupper = \{\, x\in X \mid m\ble x \text{ for some } m\in M_i \,\}$.
Figure \ref{fig:minima1} shows a subset $X_i$ (indicated by crosses)
together with its set of minima $M_i$ (circled crosses). 
In this example $X_i$ is \emph{saturated} in the sense that $X_i = M_i^\bupper$.

\begin{figure}[htbp]
  \includegraphics[height=50mm]{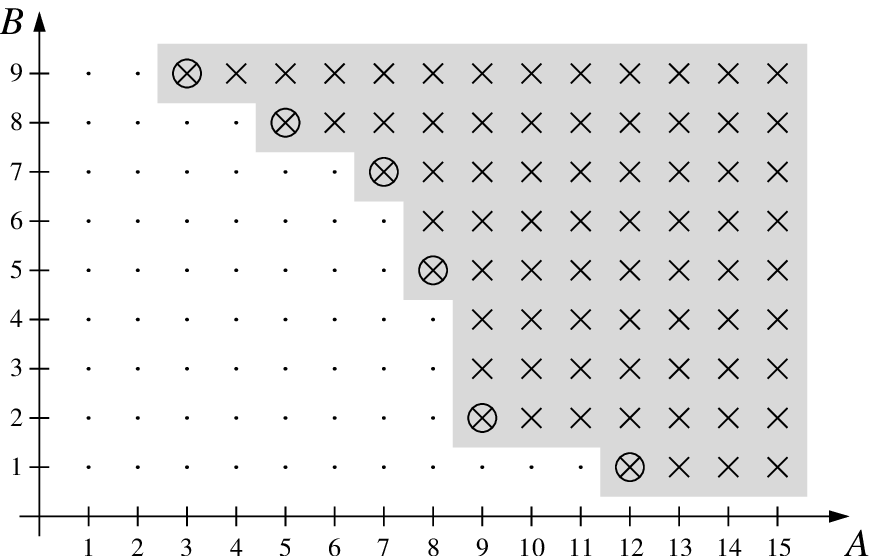}
  \caption{A subset of $A\times B$ and its minima with respect to $\ble$}
  \label{fig:minima2}
\end{figure} 

Since $f$ is monotone with respect to $\ble$,
the minimum of $f(X_i)$ is attained on $M_i$.
It thus suffices to find $x_i\in M_i$ realizing $f(x_i) = \min f(M_i)$.
We can then output $x_i$ and continue with the set 
$X_{i+1} = X_i \minus \{x_i\}$, which is again saturated.
Moreover, it is possible to construct $M_{i+1}$ directly from $M_i$,
without having to construct $X_i$ or $X_{i+1}$. 
(See Algorithm \ref{algo:BimonotoneEnumeration} below.)
Thus, instead of searching the entire set $X_i$,
we only need to keep track of $M_i$, the set of minimal elements.

\subsection{Suitable data structures} \label{sub:BimonotoneDataStructures}

According to the previous remark, the bimonotone enumeration 
algorithm will operate on two sets: $M = \Min(X_i)$ and $F = f(M)$.
The set $M$ can profitably be implemented as a list 
$\List{(a_1,b_1),(a_2,b_2),\dots,(a_m,b_m)}$ with $a_i\in A$ and $b_i\in B$.
During the algorithm, the list $M$ will always be ordered in the sense that
$a_1 < a_2 < \dots < a_m$ and  $b_1 > b_2 > \dots > b_m$,
as already indicated in Figure \ref{fig:minima2}.
We call $(a_k,b_k)$ the \emph{predecessor} of $(a_{k+1},b_{k+1})$ in $M$,
and conversely $(a_{k+1},b_{k+1})$ the \emph{successor} of $(a_k,b_k)$ in $M$.
By convention the predecessor of $(a_1,b_1)$ is $(-\infty,+\infty)$,
and the successor of $(a_m,b_m)$ is $(+\infty,-\infty)$.

Given an element $(a,b)$ in the list $M$, the required operations are:
\begin{itemize}
\item Finding the successor or predecessor of $(a,b)$ in $M$.
\item Inserting an element into $M$ right after $(a,b)$.
\item Removing $(a,b)$ from $M$.
\end{itemize}

The cost of these operations can be assumed to be $O(\log n)$, 
which is the typical cost for storing and handling 
one of the elements of the set $X = A\times B$ of size $n$. 
In particular, the cost is independent of the size $m = |M|$. 
For details on bidirectional lists see Knuth \cite[\textsection2.2.5]{Knuth:vol1},
or any other textbook on algorithms and data structures.

As before, the set $F=f(M)$ will be implemented as a priority queue 
containing the values $f(a,b)$ for all $(a,b)$ in $M$.
It is recommendable to store $f(a,b)$ together 
with a pointer to the element $(a,b)$ in the list $M$.
This allows us to extract $(a,b)$, and, moreover, 
we can directly address $(a,b)$ in $M$ without searching the list.
For notational convenience we will not explicitly 
mention this pointer in the sequel.

\subsection{The bimonotone enumeration algorithm} \label{sub:BimonotoneAlgorithm}

Having suitable data structures at our disposal, it is 
an easy matter to formalize bimonotone enumeration 
(Algorithm \ref{algo:BimonotoneEnumeration}).

\begin{algorithm}[htbp]
  \caption{\quad Sorted enumeration for a bimonotone function} 
  \label{algo:BimonotoneEnumeration}
  
  \begin{specification}
    \REQUIRE  a bimonotone function $f \colon A \times B \to Z$
              %% where $A$ and $B$ are finite ordered sets
    \OUTPUT   an enumeration of $A \times B$, monotone with respect to $f$
  \end{specification}
  
  \smallskip\hrule
  
  \begin{algorithmic}[1]
    \STATE Initialize $M \gets \List{ (\amin,\bmin) }$
           and $F \gets \Queue{ f(\amin,\bmin) }$.
    \WHILE{$F$ is non-empty}
      \STATE Remove a minimal element $f(a,b)$ from $F$ and output $(a,b)$.
      \STATE Let $(a^*,b_*)$ be the successor of $(a,b)$ in the list $M$.
      \IF{ $\succ{a}<a^*$ }
        \STATE Insert $(\succ{a},b)$ into $M$ right after $(a,b)$ 
               and insert $f(\succ{a},b)$ into $F$.
      \ENDIF
      \STATE Let $(a_*,b^*)$ be the predecessor of $(a,b)$ in the list $M$.
      \IF{ $\succ{b}<b^*$ }
        \STATE Insert $(a,\succ{b})$ into $M$ right after $(a,b)$ 
               and insert $f(a,\succ{b})$ into $F$.
      \ENDIF
      \STATE Remove $(a,b)$ from the list $M$.
    \ENDWHILE
  \end{algorithmic}
\end{algorithm}

The only subtlety of this algorithm is updating the list $M$.
We want to remove $(a,b)$, of which we know that it is a minimal element of $X_i$.
The set of elements strictly greater than $(a,b)$ is given by
\[
\{ (a,b) \}^\bupper \minus \{ (a,b) \} = \{\, (a,\succ{b}), (\succ{a},b) \,\}^\bupper.
\]
Hence, removing $(a,b)$ creates at most two new minima, $(a,\succ{b})$ and $(\succ{a},b)$.
It is easy to check whether they are actually minimal for $X_i \minus \{(a,b)\}$: 
since our list $M$ of minima is ordered, it suffices to compare $(a,\succ{b})$ to 
the predecessor $(a_*,b^*)$, and $(\succ{a},b)$ to the successor $(a^*,b_*)$.

To illustrate the different possibilities, 
we consider Figure \ref{fig:minima2} again.
The following table indicates, for each possible minimum $(a,b)$,
how the list $M$ has to be modified in order to obtain a new ordered 
list of minima satisfying $M^\bupper = X_i \minus \{(a,b)\}$:

\begin{center}
\smallskip
\renewcommand{\arraystretch}{1.2}
\begin{tabular}{|c||c|c|c||c|c|c|}
\hline
$(a,b)$        & $(a^*,b_*)$          & $(\succ{a},b)$   & insert?
               & $(a_*,b^*)$          & $(a,\succ{b})$   & insert? \\
\hline
$(3,9)$        & $(5,8)$              & $(4,9)$          & yes
               & $(-\infty,+\infty)$  & $(3,+\infty)$    & no \\
$(5,8)$        & $(7,7)$              & $(6,8)$          & yes
               & $(3,9)$              & $(5,9)$          & no \\
$(7,7)$        & $(8,5)$              & $(8,7)$          & no
               & $(5,8)$              & $(7,8)$          & no \\
$(8,5)$        & $(9,2)$              & $(9,5)$          & no 
               & $(7,7)$              & $(8,6)$          & yes \\
$(9,2)$        & $(12,1)$             & $(10,2)$         & yes
               & $(8,5)$              & $(9,3)$          & yes \\
$(12,1)$       & $(+\infty,-\infty)$  & $(13,1)$         & yes
               & $(9,2)$              & $(12,2)$         & no \\
\hline
\end{tabular}
\smallskip
\end{center}

\begin{lemma}
  Algorithm \ref{algo:BimonotoneEnumeration} is correct:
  if $A$ and $B$ are finite ordered sets and 
  $f\colon A\times B\to Z$ is a bimonotone map,
  then Algorithm \ref{algo:BimonotoneEnumeration} produces 
  a stream enumerating all pairs $(a,b) \in A \times B$ 
  such that the values $f(a,b)$ appear in increasing order.
\end{lemma}

\begin{proof}
  At the beginning of the $i$-th iteration of the algorithm
  we denote $M$ by $M_i$, and $F$ by $F_i$, and 
  the set of remaining parameters by $X_i := M_i^\bupper$.
  The initialization states that $M_1=\List{(\amin,\bmin)}$
  and $F_1 = \Queue{f(\amin,\bmin)}$, so $X_1=A\times B$.
  
  By induction we can assume that the set $X_i$ is of size $n-i+1$
  and saturated, with $M_i = \Min(X_i)$ and $F_i = f(M_i)$.
  Furthermore, we can assume that the list representing $M_{i}$ is ordered
  in the sense that $\List{(a_1,b_1),(a_2,b_2),\dots,(a_m,b_m)}$ 
  satisfies $a_1<a_2<\dots<a_m$ and $b_1>b_2>\dots>b_m$.
  It is straightforward to verify that the $i$-th iteration 
  of our algorithm ensures the following assertions:
  
  \begin{itemize}
  \item The output $x_i$ satisfies $f(x_i) = \min F_i = \min f(M_i) = \min f(X_i)$.
  \item The set $X_{i+1} = X_i \minus \{x_i\}$ is saturated and of size $n-i$.
  \item We have $M_{i+1} = \Min(X_{i+1})$ and $F_{i+1}=f(M_{i+1})$.
  \item The new list representing the set $M_{i+1}$ is again ordered.
  \end{itemize}
  
  The algorithm stops after $n$ iterations when it reaches 
  $X_{n+1}=\emptyset$, hence $M_{n+1}=\emptyset$ and $F_{n+1}=\emptyset$.
  We conclude that the output sequence $x_1,x_2,\dots,x_n$
  is an enumeration of $A\times B$ satisfying 
  $f(x_1)\tle f(x_2)\tle \dots \tle f(x_n)$.
\end{proof}

Since the sorted enumeration algorithm outputs one element $x_i$
at each iteration, the loop is repeated exactly $n = |A|\cdot|B|$ times.
At the $i$-th iteration, the algorithm occupies memory 
of size $m_i = |M_i|$ to store the list $M_i$ and the priority queue $F_i$. 
Let $m = \max m_i$ % $m = \max\{ m_1, m_2, \dots, m_n \}$ 
be the maximum during the entire execution.

\begin{lemma} \label{lem:AsymptoticBounds}
  In order to enumerate a set $X = A \times B$ of size $n$,
  Algorithm \ref{algo:BimonotoneEnumeration} only 
  builds up a priority queue of length $m \le \min\{|A|,|B|\}$,
  which entails in particular $m \le \sqrt{n}$.
\end{lemma}

\begin{proof}
  During the enumeration algorithm the list representing $M$ 
  is always strictly increasing in $a$ and strictly decreasing in $b$.
  In particular, the projections $M\to A$ and $M\to B$ are both injective.
  The required memory $m$ is thus bounded by $\min\{|A|,|B|\}$.
\end{proof}

% Algorithm \ref{algo:BimonotoneEnumeration} takes time 
% $O(n \log n \log m)$ and requires memory $O(m\log n)$.  
% Moreover, $m$ is bounded by $\min\{|A|,|B|\}$,
% which entails in particular $m \le \sqrt{n}$.

% Applied to a semimonotone map $f \colon \N \times \N \to \Z$
% that behaves polynomially, and restricted to subsets $A = \{1,\dots,l\}$
% and $B = \{1,\dots,m\}$ with $m \ge 2$,the algorithm 
% takes time $O(n \log n \log m)$, thus $O(n \log^2 n)$, 
% and requires memory $\Theta(m \log n)$, thus $O( \sqrt{n} \log n )$.

\begin{proposition} \label{prop:AsymptoticBounds}
  Let $f \colon \N \times \N \to \Z$ be a bimonotone map 
  that behaves polynomially.  Applied to subsets $A = \{1,\dots,l\}$
  and $B = \{1,\dots,m\}$ with $m \ge 2$, Algorithm \ref{algo:BimonotoneEnumeration} 
  takes time $O(n \log^2 n)$ and requires memory $O( \sqrt{n} \log n )$.
\end{proposition}

\begin{proof}
  The loop is repeated $n$ times. The most time consuming operation 
  is updating the priority queue $F_i$ to $F_{i+1}$
  which requires time $O(\log n \log m_i)$,
  where $m_i$ is the size of the queue $F_i$ 
  and its elements are typically of size $\Theta(\log n)$.
  % and $\log n$ is the typical size of its elements.
  The total cost is time $O(n \log n \log m)$
  and memory $\Theta(m \log n)$.  
  With $m \le \sqrt{n}$ we obtain the stated bounds.
\end{proof}

\begin{example}
  The following extreme cases illustrate
  Algorithm \ref{algo:BimonotoneEnumeration}
  and the possible behaviour of the memory bound $m$.
  We consider $f \colon A \times B \to \N$
  where $A=\{1,\dots,k\}$ and $B=\{1,\dots,l\}$
  are two intervals of integers, with $k\ge l\ge 2$ say.

  The best case occurs for $f(a,b) = l a + b$,
  where Algorithm \ref{algo:BimonotoneEnumeration}
  enumerates $A\times B$ in lexicographic order.
  During the $i$-th iteration of the algorithm 
  the set of minima $M_i$ contains only $1$ or $2$ elements,
  so that $m = 2$, independent of the sizes $|A|$ and $|B|$.
  
  The worst case occurs for $f(a,b)=a+b$.
  Having enumerated all elements $x_i$ with $f(x_i)\le l$,
  the list $M$ contains exactly $l$ elements,
  namely $(1,l),(2,l-1),\dots,(l,1)$.
  Thus the upper bound $m=\min\{|A|,|B|\}$ is actually attained.
\end{example}

\subsection{Enumerating infinite sets} \label{sub:BimonotoneInfinite}

As with semimonotone enumeration, bimonotone enumeration 
can be generalized from finite to infinite sets.
The interesting point is that now both sets $A$ and $B$ 
can be infinite, and the algorithm applies without change.

\begin{theorem} \label{thm:AsymptoticBounds}
  Suppose that $A$ and $B$ are ordered sets and that
  $f\colon A\times B\to Z$ is a proper bimonotone map.
  Then Algorithm \ref{algo:BimonotoneEnumeration} produces a stream
  enumerating all pairs $(a,b)\in A\times B$ such that 
  the values $f(a,b)$ appear in increasing order.
  In order to enumerate the first $n$ values, the algorithm 
  builds up a priority queue of length at most $\sqrt{2n}+1$.
  If $f$ behaves polynomially, the algorithm takes time 
  $O(n\log^2 n)$ and requires memory $O(\sqrt{n} \log n)$.
\end{theorem}

\begin{proof}
  For every $z\in\im(f)$ the set $\{\, (a,b)\in A\times B \mid f(a,b)\le z \,\}$
  is finite, it is thus contained in some finite product $[\amin,a_1] \times [\bmin,b_1]$.
  Hence Algorithm \ref{algo:BimonotoneEnumeration} correctly enumerates 
  all parameters $(a,b)$ with $f(a,b)\le z$ as in the finite case.
  Since this is true for all $z$, the enumeration exhausts $A\times B$.
  
  Let us suppose that after $n$ outputs the list $M$ holds $m$ pairs $(a_i,b_i)$
  ordered such that $a_1 < a_2 < \dots < a_m$ and $b_1 > b_2 > \dots > b_m$.
  This obviously implies $n \ge \frac{1}{2}m(m-1)$, whence $m \le \sqrt{2n}+1$.
  The time needed for $n$ outputs is thus 
  $O(n \log n \log\sqrt{n}) = O(n\log^2 n)$, 
  while the required memory is $O(\sqrt{n} \log n)$.
\end{proof}

\begin{example} \label{exm:TypicalExponents}
  Again the worst case occurs for the map $f \colon \N \times \N \to \Z$ 
  with $f(a,b)=a+b$, where $m \sim \sqrt{2n}$.
  The best case occurs for $f(a,b) = \max\{a,b\}$, where $m \le 4$.
\end{example}

This example shows that for bimonotone enumeration 
the memory bound $m \in O(\sqrt{n})$ is best possible:
there exist bimonotone functions $f$ for which 
Algorithm \ref{algo:BimonotoneEnumeration}
actually requires temporary memory $m \sim \sqrt{2n}$.  
Notwithstanding, the algorithm performs significantly 
better on certain subclasses of bimonotone functions:
%% The following examples interpolate between the two extremes:

\begin{proposition}[separate variables] \label{prop:TypicalExponents}
  Consider $f \colon \N \times \N \to \Z$ with
  $f(a,b) = p(a) + q(b)$, where $p$ and $q$ 
  are non-decreasing polynomial functions of degree 
  $\alpha = \deg p$ and $\beta = \deg q$, respectively.
  Assuming $1 \le \alpha \le \beta$, the bimonotone enumeration 
  algorithm requires memory $m \in \Theta(n^\varepsilon)$ with 
  exponent $\varepsilon = \frac{\alpha}{\alpha+\beta}$. % \in [0,\frac{1}{2}]$.
\end{proposition} 

For example, sorted enumeration of $f(a,b) = a^3 + b^7$
requires memory $m \in \Theta(n^{3/10})$, while the a priori bound 
of Theorem \ref{thm:AsymptoticBounds} only tells us $m \in O(n^{1/2})$.

\begin{proof}
  Let $f \colon \N \times \N \to \N$ be defined by 
  $f(a,b) = a^\alpha + b^\beta$ with $1 \le \alpha \le \beta$.
  The general case $f(a,b) = p(a) + q(b)$ works essentially 
  the same, but notation is more cumbersome.

  Suppose that the $n$-th output $x_n$ has attained the level $f(x_n)=z$,
  and the list $M$ holds $m$ parameters $(a_1,b_1),\dots,(a_m,b_m)$.
  Then we have $a_1=0$ and $f(a_1,b_1) = b_1^{\smash\beta} \ge z$.
  On the other hand $(a_1,b_1-1)$ has already been output, 
  which means $(b_1-1)^\beta \le z$.  
  We conclude that $\sqrt[\beta]{z} \le b_1 \le \sqrt[\beta]{z} + 1$, 
  whence $b_1 \sim \sqrt[\beta]{z}$.  
  Analogously $\sqrt[\alpha]{z} \le a_m \le \sqrt[\alpha]{z} + 1$,
  whence $a_m \sim \sqrt[\alpha]{z}$.
  This situation is depicted in Figure \ref{fig:convex}.
  In the sequel we set $a:= a_m$ and $b:= b_1$. 

  \begin{figure}[htbp]
    \includegraphics[scale=1.0]{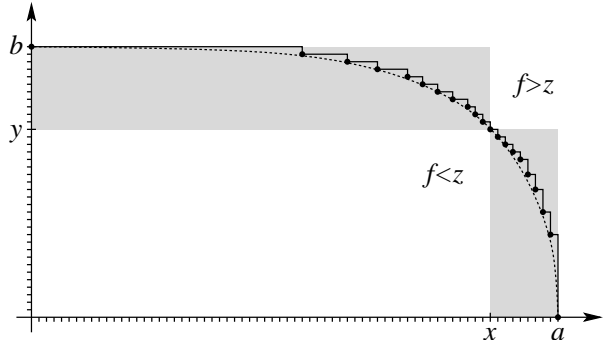}
    \caption{Estimating the size of the set $\Min_{\ble} \{ a^\alpha + b^\beta \ge z \}$}
    \label{fig:convex}
  \end{figure} 

  The upper bound $n \le (a+1) (b+1)$ is clear. 
  Since $f$ is convex, we also have the lower bound $n > \frac{1}{2} a b$.
  To see this, apply Pick's theorem to count integer points 
  in the triangle $\Delta = [\,(0,0), (a,0), (0,b)\,]$.
  Both inequalities together imply that 
  $n \in \Theta( z^{(\alpha+\beta)/\alpha\beta} )$,
  or equivalently, $z \in \Theta( n^{\alpha\beta/(\alpha+\beta)} )$.
  
  The upper bound $m \le b + 1$ %% $m \le b + 1 \sim \sqrt[\beta]{z}$ 
  is clear, and it remains to establish a lower bound.  
  We will assume $\alpha < \beta$.  
  (The symmetric case $\alpha = \beta$ is easier and will be 
  examined more closely in Example \ref{exm:gamma} below.)
  There exists a unique point $(x,y) \in \R_+^2$ on the contour
  whose normal vector points in the direction $(1,1)$:
  this is the solution of $x^\alpha + y^\beta = z$ 
  and $\alpha x^{\alpha-1} = \beta y^{\beta-1}$.
  It is easy to see that $m \ge b - \lceil y \rceil$ and 
  $m \sim (a - x) + (b - y)$ as indicated in Figure \ref{fig:convex}.
  We have $x = c y^{(\beta-1)/(\alpha-1)}$ 
  with $c = \sqrt[\alpha-1]{\beta/\alpha}$,
  whence $c^\alpha y^{\alpha(\beta-1)/(\alpha-1)} + y^\beta = z$.
  Since $\alpha(\beta-1)/(\alpha-1) > \beta$ 
  we deduce that $y \in o( \sqrt[\beta]{z})$. 
  The bounds $b - \lceil y \rceil \le m \le b+1$
  thus entail $m \sim \sqrt[\beta]{z}$,
  whence $m \in \Theta( n^{\alpha/(\alpha+\beta)} )$.
\end{proof}

% Example \ref{exm:TypicalExponents} shows 
% that for the bimonotone enumeration algorithm 
% the memory bound $m \in O(\sqrt{n})$ is best possible:
% there exist bimonotone functions that actually require 
% temporary memory $m \in \Theta(\sqrt{n})$.  Notwithstanding, 
% the algorithm performs significantly better on certain subclasses 
% of bimonotone functions, where $\varepsilon < \frac{1}{2}$.

We remark that in the above examples semimonotone 
enumeration achieves the same asymptotic bounds.  
This warrants a more detailed analysis, which we endeavour next.

\section{Asymptotic complexity} \label{sec:AsymptoticComplexity}

We are now ready to address the crucial question:
is bimonotone enumeration (Algorithm \ref{algo:BimonotoneEnumeration})
better than semimonotone enumeration (Algorithm \ref{algo:SesquimonotoneEnumeration})?
We shall compare the size $m$ of the priority queue built up during the algorithm.
% (We neglect additional costs due to pointer management in doubly linked lists.  
% Execution time also depends on $m$, but only logarithmically.)
The test class consists of all proper bimonotone functions
$f \colon \N \times \N \to \Q$, which is where both algorithms apply.  
First of all, the following observation is worth emphasizing:

\begin{remark}
  Bimonotone enumeration is at least as good as semimonotone enumeration.
  More explicitly, both algorithms have to trace the contour of the finite set
  \[
  \{ f \tle z \} := \{ (a,b) \in A \times B \mid f(a,b) \tle z \}
  \]
  and construct the set of minima of the complement, $\Min \{ f > z \}$.
  To this end semimonotone enumeration uses the partial order 
  $(a,b) \sle (a',b')$ defined by $a \tle a'$ and $b = b'$.
  (Here we are ordering with respect to $a$ for fixed $b$;
  since $f$ is bimonotone, we could also order with respect 
  to $b$ for fixed $a$, whichever is advantageous.)
  Bimonotone enumeration uses the partial order $(a,b) \ble (a',b')$ 
  defined by $a \tle a'$ and $b \tle b'$.  This entails the inclusion
  \[
  \Min_{\ble} \{ f > z \} \subset \Min_{\sle} \{ f > z \} .
  \] 
  This means that the priority queue for bimonotone enumeration 
  is a subset of the queue for semimonotone enumeration,
  and consequently the required memory is less or equal.
 
  At this point we should clarify a possible ambiguity.
  Both Algorithms \ref{algo:SesquimonotoneEnumeration} 
  and \ref{algo:BimonotoneEnumeration} have to choose
  \emph{one} minimal element of the priority queue.
  In order to disambiguate multiple minima, we choose 
  the one with minimal $B$-coordinate.  This ensures 
  that it belongs to both $\Min_{\ble} \{ f > z \}$
  and $\Min_{\sle} \{ f > z \}$, and the 
  inclusion propagates inductively.
\end{remark}

Whether the bimonotone algorithm can achieve 
a significant improvement depends on the function $f$.
Let us begin with a trivial example where no savings are possible:
  
\begin{example}[linear contour]
  Consider $f \colon \N \times \N \to \N$ defined by
  $f(a,b) = (a + b)^\gamma$ with $\gamma \ge 1$.
  In this case $\Min_{\ble} \{ f > z \} = \Min_{\sle} \{ f > z \}$
  is given by the line $a + b = 1 + \lfloor \sqrt[\gamma]{z} \rfloor$.
\end{example}

In general, however, the inclusion 
$\Min_{\ble} \{ f > z \} \subset \Min_{\sle} \{ f > z \}$ is strict.  
Generally speaking, bimonotone enumeration adapts better to the contour and 
achieves savings whenever the contour deviates from being a straight line.  
We now quantify this observation.

\subsection{Polynomial functions}

Consider $f \colon \N \times \N \to \Q$ defined by 
a polynomial $f(a,b) = \sum_{i,j} c_{ij} a^i b^j$ 
with rational coefficients $c_{ij} \ge 0$.
This condition ensures that $f$ is bimonotone.

Let $f_1(a) = f(a,0)$ and $f_2(b) = f(0,b)$ be the induced polynomial functions 
on the axes, and set $\alpha := \deg f_1$ and $\beta := \deg f_2$.
We assume that $\alpha,\beta \ge 1$, which ensures that $f$ is proper.
Without loss of generality we can also assume that $\alpha \le \beta$.

Let $\gamma := \deg f = \max \{ i+j \mid c_{ij} \ne 0 \}$
be the total degree of $f$.  We have $\alpha \le \beta \le \gamma$.

We denote by $n := \sharp \{ f \tle z \}$ the number of values of $f$
up to some level $z$, and by $m := \sharp \Min \{ f > z \}$ 
the length of the priority queue at level $z$.  
% We tacitly assume that $m$ is non-decreasing in $z$, so that 
% $m$ represents the memory needed to enumerate $\{ f \tle z \}$.

% In order to enumerate the set $\{ f \tle z \}$, 
% semimonotone enumeration (Algorithm \ref{algo:SesquimonotoneEnumeration})
% requires memory $m \in \Theta(\sqrt[\beta]{z})$ whereas 
% bimonotone enumeration (Algorithm \ref{algo:BimonotoneEnumeration})
% requires memory $m \in O(\sqrt[\gamma]{z})$.
  
\begin{proposition} \label{prop:CompareMemory}
  Semimonotone enumeration of the set $\{ f \tle z \}$ 
  requires memory $m \in \Theta(\sqrt[\beta]{z})$ whereas 
  bimonotone enumeration requires memory $m \in O(\sqrt[\gamma]{z})$.
\end{proposition}

Whenever $\beta < \gamma$, bimonotone enumeration 
is thus significantly better than semimonotone enumeration.
As an example, for $f(a,b) = a^4 + a^3 b^4 + b^5$ 
semimonotone enumeration requires memory $m \in \Theta(\sqrt[5]{z})$, 
while bimonotone enumeration requires only $m \in O(\sqrt[7]{z})$.

\begin{proof}
  Assuming $\beta \ge \alpha$, it is advantageous 
  to sort by $a$ in the semimonotone enumeration algorithm.
  In this case we see that $m = 2+b$ where $b$ satisfies $f_2(b) \le z < f_2(b+1)$.  
  We have $f_2(b) \sim c b^\beta$ with some leading coefficient $c > 0$,
  whence $m \sim \sqrt[\beta]{z/c}$.

  \begin{figure}[htbp]
    \includegraphics[scale=1.0]{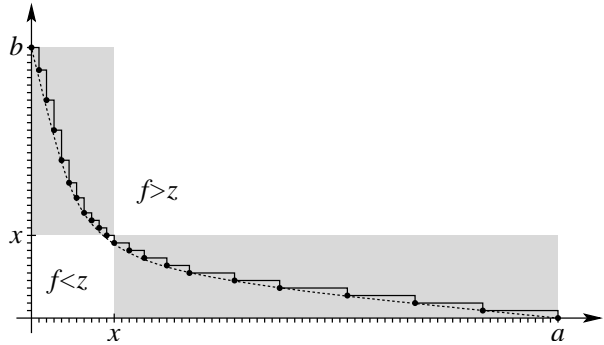}
    \caption{Estimating the size of the set $\Min_{\ble} \{ f \ge z \}$}
    \label{fig:concave}
  \end{figure} 

  Evaluating $f$ on the diagonal, we find 
  $f(x,x) = d_0 + d_1 x + \dots + d_\gamma x^\gamma$
  with non-negative coefficients $d_k \ge 0$ and $d_\gamma > 0$.
  For $x = 1 + \lfloor \sqrt[\gamma]{z/d_\gamma} \rfloor$ we obtain $f(x,x) > z$.  
  This is illustrated in Figure \ref{fig:concave}, where the dotted line corresponds 
  to $f = z$, and black dots represent the elements of $\Min_{\ble} \{ f \ge z \}$.
  We conclude that $m \le 1+2x$, whence $m \in O(\sqrt[\gamma]{z})$.
\end{proof}

\subsection{Asymptotic bounds}

In order to express the required memory $m$ in terms of 
the number $n$ of enumerated values, we wish to relate $n$ and $z$.
For $z \to \infty$ we can replace counting points $(a,b) \in \N^2$ 
satisfying $f(a,b) \le z$ by the Lebesgue measure of the set 
$\{ (x,y) \in \R_+^2 \mid f(x,y) \le z \}$.

\begin{proposition} \label{prop:TropicalExponent}
  Let $f \colon \R_+^2 \to \R_+$ be a polynomial function given by 
  \[
  f(x,y) = \sum_{(i,j) \in K} c_{ij} x^i y^j \qquad
  \text{ with $c_{ij} > 0$ for all indices $(i,j) \in K$.}
  \]
  With $f$ we associate the convex polygon 
  $D = \{ (u,v) \in \R^2 \mid i u + j v \le 1 \text{ for all } (i,j) \in K \}$.
  Suppose that $f$ is proper in the sense that for all $z \in \R_+$ 
  the set 
  \[
  \{ f \le z \} = \{ (x,y) \in \R_+^2 \mid f(x,y) \le z \}
  \]
  is compact.  Then its Lebesgue measure satisfies
  $\lambda(\{ f \le z \}) \in \Theta( z^\delta \log(z)^d )$ 
  where 
  \[
  \delta := \max\{ u + v \mid (u,v) \in D \}
  \]
  and $d$ is the dimension of the set where this maximum 
  is attained: either $d=0$ for a vertex, or $d=1$ for a segment.
  (See Figure \ref{fig:tropical} below for examples.)
\end{proposition}

The proof is a nice application of the so-called ``tropical'' approach.
The idea is to identify $\R_+ = \{ x \in \R \mid x \ge 0 \}$ 
and $\hat\R = \R \cup \{-\infty\}$ via the natural logarithm
$\log \colon \R_+ \to \hat\R$, and to formally replace 
the semiring $(\R_+,+,\cdot)$ by the semiring $(\hat\R,\max,+)$.
Of course, we have $\log( x \cdot y ) = \log x + \log y$ 
but for $\log(x+y)$ we only obtain an inequality,
\[
\max( \log x, \log y ) \le \log(x+y) \le \log 2 + \max( \log x, \log y ) .
\]
This means that $\log \colon \R_+ \to \hat\R$ is a \emph{quasi-isomorphism},
i.e.\ its failure to be an isomorphism is bounded by some constant.
For asymptotic arguments this is usually sufficient.

\begin{proof}[Proof of Proposition \ref{prop:TropicalExponent}]
  As a logarithmic analogue of $f$ we define 
  \[
  \hat{f} \colon \hat\R^2 \to \hat\R, \quad
  \hat{f}(\hat{x},\hat{y}) := \max_{(i,j) \in K}\left( i \hat{x} + j \hat{y} \right) .
  \]

  We can choose a constant $c \in \R_+$ such that $c_{ij} \ge e^{-c}$ 
  and $(\sharp K) \cdot c_{ij} \le e^{+c}$ for all $(i,j) \in K$.  
  A small calculation then shows that
  \[
  \left| \log f(x,y) - \hat{f}( \log x, \log y ) \right| \le c .
  % \qquad \text{ for all $x,y \in \R_+$.}
  \]
  % To see this notice that positivity implies the double inequality
  % \[
  % \max_{(i,j) \in K} ( c_{ij} x^i y^j ) \le f(x,y) 
  % \le (\sharp K) \cdot \max_{(i,j) \in K} ( c_{ij} x^i y^j ) .
  % \]
  % Applying the logarithm we obtain
  % \[
  % \max_{(i,j) \in K} \log(c_{ij}) + i \log{x} + j \log{y}
  % \le \log f(x,y) 
  % \le \log(\sharp K) + \max_{(i,j) \in K} \log(c_{ij}) + i \log{x} + j \log{y}
  % \]

  The measure of the set $\{ f \le z \}$ equals the integral 
  over the associated indicator function $[ f \le z ]$ 
  and we can apply the change of variables 
  $\hat{x} = \log x$, $\hat{y} = \log y$, $\hat{z} = \log z$:
  \[
  F(z) := \int_{\R_+^2} \left[ f(x,y) \le z \right] \, dx \, dy 
  = \int_{\R^2} \left[ \log f( \hat{x}, \hat{y} ) \le \hat{z} \right] 
  \exp( \hat{x} + \hat{y} ) \, d\hat{x} \, d\hat{y} .
  \]

  It is easier to calculate this integral with $\hat{f}$ 
  instead of $f$, so let us do this first.
  Since $\hat{f}$ is homogeneous, we perform another change of variables
  $\hat{x} = u \hat{z}$ and $\hat{y} = v \hat{z}$ to obtain:
  \[
  \hat{F}(z) := \int_{\R^2} \left[ \hat{f}( \hat{x}, \hat{y} ) \le \hat{z} \right] 
  \exp( \hat{x} + \hat{y} ) \, d\hat{x} \, d\hat{y} 
  = \hat{z}^2 \int_{\R^2} \left[ \hat{f}( u, v ) \le 1 \right] 
  \exp( u\hat{z} + v\hat{z} ) \, du \, dv .
  \]
  We are now integrating over the convex polygon 
  $D := \{ (u,v) \in \R^2 \mid \hat{f}( u, v ) \le 1 \}$.
  The asymptotic behaviour of $\log \hat{F}(z)$ is easy to understand:
  for $z \to \infty$ we obtain
  \[
  \frac{\log \hat{F}(z)}{\log z} = \log\left[ \hat{z}^{2/\hat{z}} 
    \left( \int_D \exp( u + v )^{\hat{z}} \, du \, dv \right)^{1/\hat{z}} \right]
  \to \delta.
  \]
  For $z \to \infty$ the first factor $\hat{z}^{2/\hat{z}} \to 1$ plays no r\^ole.
  The remaining factor is the $\hat{z}$-norm $\| \exp( u + v ) \|_{\hat{z}}$
  and tends to the sup-norm $\| \exp( u + v ) \|_\infty = \exp(\delta)$ 
  for $\hat{z} \to \infty$.  

  This shows that $\log \hat{F}(z) \sim \delta \log z$,
  but does not yet suffice to imply $\hat{F}(z) \sim z^\delta$
  for $z \to \infty$.  We thus have a closer look at the quotient
  \[
  \frac{\hat{F}(z)}{z^\delta} = \log(z)^2 \int_D z^{(u+v-\delta)}  \, du \, dv .
  \]
  % For $(u,v) \in D$ we have $u+v-\delta \le 0$, 
  % with equality precisely where $u+v$ attains its maximum.
  % For $z \to \infty$ the integral is equivalent to $\const \cdot \log(z)^{-2}$ 
  % if this maximum is attained in a single vertex, 
  % otherwise it is equivalent to $\const \cdot \log(z)^{-1}$.  
  % The calculations are omitted.  
  %
  We change variables $u = \frac{\delta-t}{2}-s$ and $v = \frac{\delta-t}{2} + s$
  so that $u+v = \delta-t$ and $du \, dv = ds \, dt$:
  \[
  \int_D z^{(u+v-\delta)} \, du \, dv = \int_0^\infty \ell(t) z^{-t} dt
  \]
  where $\ell(t)$ is the length of the segment $\{ (u,v) \in D \mid u+v=\delta-t \}$.
  Since $D$ is a polygon, there exist $a_0,a_1 \ge 0$ 
  and $T>0$ such that $\ell(t) = a_0 + a_1 t$ for all $t \in [0,T]$.
  We thus find 
  \[
  \int_0^\infty \ell(t) z^{-t} dt 
  \sim a_0 \log(z)^{-1} + a_1 \log(z)^{-2} \qquad\text{for $z \to \infty$.}
  \]
  Notice that $a_0 = 0$ if and only if 
  the maximum $u+v=\delta$ is attained in a single vertex.
  We thus obtain $\hat{F}(z) \in \Theta( z^\delta \log(z)^d )$
  where $\delta$ is the maximum of $u+v$ on $D$
  and $d$ is the dimension of the maximising set.
  Since $\hat{F}(e^{-c} z) \le F(z) \le \hat{F}(e^{+c} z)$,
  we conclude that $F(z) \in \Theta( z^\delta \log(z)^d )$.
\end{proof}

% We change variables $u = \frac{\delta-t}{2}-s$ and $v = \frac{\delta-t}{2} + s$
% so that $u+v = \delta-t$ and $du \, dv = ds \, dt$.
% We thus obtain $\int_D z^{(u+v-\delta)} \, du \, dv = \int_0^\infty \ell(t) z^{-t} dt$
% where $\ell(t)$ is the length of the segment $\{ (u,v) \in D \mid u+v=\delta-t \}$.
% Since $D$ is a convex polygon, there exist $a_0,a_1 \ge 0$ 
% and $T>0$ such that $\ell(t) = a_0 + a_1 t$ for all $t \in [0,T]$.
% We find that $\int_D z^{(u+v-\delta)}  \, du \, dv \sim a_0 \log(z)^{-1} + a_1 \log(z)^{-2}$,
% with $a_0 = 0$ if and only if the maximum is attained in a single vertex.

\begin{remark}
  It is clear that the proposition and its proof generalize
  to proper polynomial functions $f \colon \R_+^n \to \R_+$
  with non-negative coefficients, in any number $n$ of variables.
  We have concentrated on $n=2$, which is the case of interest to us here.
\end{remark}

\begin{example}
  For $f(x,y) = x^4 + y^5$ the set $\{ \hat{f} \le 1 \}$ 
  is depicted in Figure \ref{fig:tropical} on the left.
  Here we obtain $\delta = \frac{9}{20}$ and $d=0$
  because the maximum is attained in a single vertex.

  The figure in the middle shows $\{ \hat{f} \le 1 \}$ 
  for $f(a,b) = a^4 + a^3 b^3 + b^5$.  Here $\delta = \frac{1}{3}$
  and $d = 1$, because the maximum is attained on a segment,
  so that $n \in \Theta( z^{1/3} \log z )$.

  \begin{figure}[htbp]
    \includegraphics[scale=0.9]{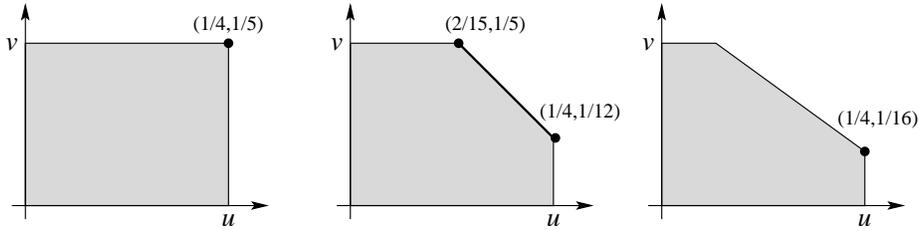}
    \caption{Maximizing $u+v$ under the constraint $\hat{f}(u,v) \le 1$}
    \label{fig:tropical}
  \end{figure} 
  
  The figure on the right shows $\{ \hat{f} \le 1 \}$ 
  for $f(a,b) = a^4 + a^3 b^4 + b^5$.
  Here we find $\delta = \frac{5}{16}$,
  which means that $z \in \Theta( n^{16/5})$.
  According to Proposition \ref{prop:CompareMemory},
  semimonotone enumeration requires memory 
  $m \in \Theta( z^{1/5} )$, whence
  $m \in \Theta( n^{16/25} )$.  

  Notice in particular that $\frac{16}{25} > \frac{1}{2}$.
  This illustrates that, unlike bimonotone enumeration, 
  semimonotone enumeration cannot guarantee the memory bound 
  $m \in O( \sqrt{n} )$. % for all bimonotone functions:
\end{example}

\begin{corollary}
  The semimonotone enumeration algorithm guarantees the memory bound $m \le n+1$, 
  and as a uniform bound the exponent $1$ is best possible:
  enumerating the values of $f(a,b) = a^\alpha + a^\alpha b^\beta + b^\beta$ 
  with $\alpha < \beta$ requires memory $m \in \Theta(n^{\alpha/\beta})$.
  \qed
\end{corollary}

\begin{proof}
  According to Proposition \ref{prop:TropicalExponent}
  the number of enumerated values up to level $z$ is
  $n \in \Theta( z^\delta )$ with $\delta = 1/\alpha$,
  and thus $z \in \Theta( n^\alpha )$.
  According to Proposition \ref{prop:CompareMemory}
  the required memory is $m \in \Theta( z^{1/\beta} )$.
  We conclude that $m \in \Theta(n^{\alpha/\beta})$.
\end{proof}

\subsection{Constant factors}

Proposition \ref{prop:CompareMemory} exhibits many polynomial functions
where bimonotone enumeration is clearly worth the effort.
Depending on the envisaged application and the given function $f$, 
a finer analysis and a more modest conclusion may be necessary:

\begin{example}
  Consider polynomials of the form $f(a,b) = p(a) + q(b)$, for which 
  semimonotone enumeration was initially devised \cite{Ekl:1996,Bernstein:2001}.
  We obtain $n \in \Theta(z^\delta)$ with $\delta = \frac{1}{\alpha} + \frac{1}{\beta}$, 
  as already remarked in the proof of Proposition \ref{prop:TypicalExponents}.  
  Assuming $\alpha \le \beta$, bimonotone and semimonotone enumeration 
  both require memory $m \in \Theta(n^\varepsilon)$ with $\varepsilon = \frac{\alpha}{\alpha+\beta}$.
\end{example}

Even if memory requirements are of the same order of magnitude, 
we can usually expect to gain a constant factor with the bimonotone algorithm:

\begin{example} \label{exm:gamma}
  Reconsider $f \colon \N \times \N \to \Z$ defined by
  $f(a,b) = a^\gamma + b^\gamma$ with $\gamma > 1$.
  In this case semimonotone enumeration requires memory 
  $m \sim \sqrt[\gamma]{z}$, whereas bimonotone enumeration 
  requires memory $m \sim c_\gamma \cdot \sqrt[\gamma]{z}$
  with a factor $c_\gamma = 2 (1-\sqrt[\gamma]{1/2} ) < 1$.
\end{example}

\begin{proof}
  For $a = b = \lfloor \sqrt[\gamma]{z} \rfloor$ 
  we have $f_1(a) = f_2(b) \le z$ and $f_1(a+1) = f_2(b+1) > z$, 
  and semimonotone enumeration requires memory
  $\sharp \Min_{\sle} \{ f > z \} = 2 + b \sim \sqrt[\gamma]{z}$.

  Choosing $x = y = \lfloor \sqrt[\gamma]{z/2} \rfloor$ 
  we find $f(x,y) \le z$ and $f(x+1,y+1) > z$.  
  As indicated in Figure \ref{fig:convex}, 
  we have $m \sim (a-x) + (b-y) \sim c_\gamma \cdot \sqrt[\gamma]{z}$.
  %
  % As indicated in Figure \ref{fig:convex}, 
  % bimonotone enumeration satisfies the bounds
  % \[
  % (a-x) + (b-x) 
  % \quad\le\quad 
  % \sharp \Min_{\ble} \{ f > z \} 
  % \quad\le\quad 
  % (a-x) + (b-x) + 2 .
  % \]
  % This implies $m \sim c_\gamma \cdot \sqrt[\gamma]{z}$ as claimed.
\end{proof}

Though less impressive, for practical applications even 
a constant factor may be a welcome improvement: 
reducing memory consumption means that we can scale 
to considerably larger problems before running out of RAM.
In our example we have $c_2 \approx 0.59$, $c_3 \approx 0.41$,
$c_4 \approx 0.32$, $c_5 \approx 0.26$, 
and $c_\gamma \to 0$ for $\gamma \to \infty$.

%%%%%%%%%%%%%%%%%%%%%%%%%%%%%%%%%%%%%%%%%%%%%%%%%%%%%%%%%%%%%%%%%%%%%%%%%%%%%

\section{Parallelization} \label{sec:Parallelization}

Let us reconsider the application of sorted enumeration 
to a diophantine equation $f(a,b)=g(c,d)$, where 
$f, g \colon \N\times\N\to\Z$ are proper bimonotone functions.
Suppose we are looking for solutions $x=(a,b)$, $y=(c,d)$ % in $\N\times\N$ 
with values in some large interval $z_{\min} \le f(x) = g(y) < z_{\max}$.
This problem can be split into $s$ independent subproblems, namely 
searching solutions with $z_{k-1} \le f(x) = g(y) < z_k$, % for $k=1,\dots,s$,
where $z_{\min} = z_0 < z_1 < z_2 < \dots < z_s = z_{\max}$ 
is a subdivision of our search interval. 
This allows us to distribute the search 
on several computers in parallel.

\subsection{The initialization algorithm} \label{sub:Initialization}

To put the parallelization idea into practice, 
Algorithm \ref{algo:ConstructMinima}, stated below, 
initializes the enumeration stream to begin at level $z$.
Graphically speaking, it traces the contour of $X = \{ f \ge z \}$
in order to determine the set of minimal elements $M = \Min X$. 
From $M$ we can then immediately build up the priority queue $F = f(M)$.

As usual we require that $f\colon A\times B\to Z$ be a proper bimonotone map.  
For simplicity we first assume that both $A$ and $B$ are infinite.  
(We will treat the general case in the next paragraph.)
As before the successor function is denoted by 
$a \mapsto \succ{a}$ and  $b \mapsto \succ{b}$, respectively.
We also use the predecessor function, denoted by $b \mapsto \pred{b}$.

\begin{algorithm}[htbp]
  \caption{\quad Constructing the set of minima on $X = A \times B$}
  \label{algo:ConstructMinima}
  
  \begin{specification}
    \REQUIRE  a proper bimonotone function $f \colon A \times B \to Z$
              %% where $A$ and $B$ are infinite ordered sets isotonic to the natural numbers
    \INPUT    a level $z \in Z$
    \OUTPUT   the list of minima $\Min \{\, x \in X \mid f(x) \tge z \,\}
              = \List{(a_1,b_1),(a_2,b_2),\dots,(a_m,b_m)}$
    \ENSURE   $a_1 < a_2 < \dots < a_m$ and  $b_1 > b_2 > \dots > b_m$
  \end{specification}

  \smallskip\hrule
  
  \begin{algorithmic}[1]
    
    \STATE Initialize $M \gets \emptyset$ and $a \gets \amin$, $b \gets \bmin$
    \WHILEDO{ $f(a,b) < z$ }{ $b \gets \succ{b}$ }
    \STATE Insert $(a,b)$ at the end of the list $M$;
           continue with $a \gets \succ{a}$, $b \gets \pred{b}$

    \WHILE{ $b \tge \bmin$ }
      \WHILEDO{ $f(a,b) < z$ }{ $a \gets \succ{a}$ }
      \WHILEDO{ $b > \bmin$ and $f(a,\pred{b}) \ge z$ }{ $b\gets \pred{b}$ }
      \STATE Insert $(a,b)$ at the end of the list $M$;
             continue with $a \gets \succ{a}$, $b \gets \pred{b}$
    \ENDWHILE

    \RETURN $M$
    
  \end{algorithmic}
\end{algorithm}

The reader is invited to apply Algorithm \ref{algo:ConstructMinima} 
to the example given in Figure \ref{fig:minima2}, 
in order to see how it traces the contour of $X = \{ f \ge z \}$. 
By the way, the method applies to any set $X \subset A \times B$ 
that is saturated and has finite complement.
We shall give a detailed proof in the more general situation
of Algorithm \ref{algo:BimonotoneMinima} below. 

\begin{remark}
  The loop in line 2 determines
  $b \gets \min \{\, b\in B \mid f(\amin,b) \tge z \,\}$.
  This could, of course, be improved by replacing the linear search
  with a binary search, provided that $b$ can easily be incremented 
  and decremented by integer values. The same holds for the loops
  in lines 5 and 6. This optimization is straightforward 
  to implement whenever the application requires it.
\end{remark}

\begin{remark}
  Let $M = \List{(a_1,b_1),(a_2,b_2),\dots,(a_m,b_m)}$ 
  be the list of minima. Then building a priority queue 
  from $M$ requires time $O(m \log n)$.
  Moreover, let $k = \#[\amin,a_m]$ and $l = \#[\bmin,b_1]$, 
  with $k \ge l$, say.  Then $n \ge k \ge l \ge m$.
  Constructing the list $M$ itself requires time $O(k \log n)$
  using linear search, and $O(m \log^2 n)$ using binary search.
  We cannot expect to do much better, because constructing 
  a list of length $m$ requires at least $m$ iterations.
\end{remark}

\subsection{Applications} \label{sub:ParallelApplications}

Having initialized $M$ and $F$, we can apply 
the bimonotone enumeration algorithm to produce a sorted 
enumeration $x_1,x_2,\dots$ of the set $\{ f \ge z_{k-1} \}$.
Applying the same method to $g$, we can produce 
a sorted enumeration $y_1,y_2,\dots$ of $\{ g \ge z_{k-1} \}$.
We can thus search for solutions $f(x) = g(y)$ 
starting at level $z_{k-1}$ and ending at level $z_k$.

\smallskip\textit{Expected speed-up.}
Concerning time requirements, initialization entails a reasonably small 
overhead, so we can expect an amortized speed-up by a factor $s$.
For each $k=1,\dots,s$, computer number $k$ manages 
its own priority queues of length $O( \sqrt{n} )$.
in order to produce enumeration streams for $f$ and $g$,
with values ranging from $z_{k-1}$ to $z_k$.
As before, advancing from position $n$ to position $n+1$ 
takes time $O(\log^2 n)$. % and memory $O( \sqrt{n} \log n) $.

\smallskip\textit{Robustness.}
The initialization procedure is already very useful on a 
single computer, since it can make implementations much more robust: 
it is possible to continue searching, without much loss, 
after a shut-down or a power failure.
This is particularly important when carrying out a long-term search.

%%%%%%%%%%%%%%%%%%%%%%%%%%%%%%%%%%%%%%%%%%%%%%%%%%%%%%%%%%%%%%%%%%%%%%%%%%%%%

\section{Enumerating bimonotone domains} \label{sec:BimonotoneDomains}

Suppose we want to enumerate the values of a \emph{symmetric} 
bimonotone function $f \colon \N \times \N \to \Z$, that is, 
$f(a,b) = f(b,a)$ for all $a,b \in \N$. % $(a,b) \in \N \times \N$.
It is often desirable to enumerate only pairs $(a,b)$ with $a \tge b$.
%% since $(b,a)$ is considered equivalent with $(a,b)$. 
In other words, we wish to \emph{restrict} $f$ to the domain 
$X = \{\, (a,b) \in \N \times \N \mid a \tge b \,\}$ and 
enumerate only the values of $f \colon X \to \Z$.
Figure \ref{fig:minima3} shows a possible 
configuration during bimonotone enumeration.

\begin{figure}[htbp]
  \includegraphics[height=50mm]{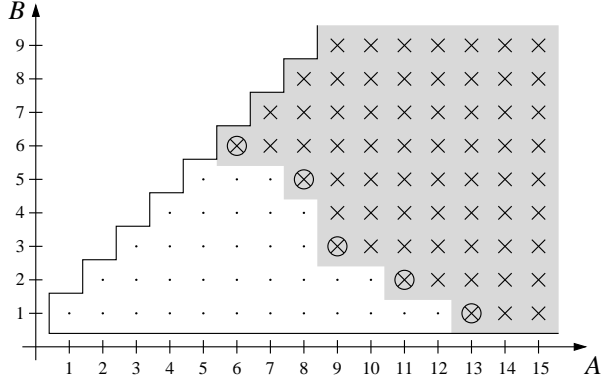}
  \caption{Enumerating the domain
    $X = \{\, (a,b) \in \N \times \N \mid a \tge b \,\}$}
  \label{fig:minima3}
\end{figure} 

It is straightforward to adapt bimonotone
enumeration (Algorithm \ref{algo:BimonotoneEnumeration})
and initialization (Algorithm \ref{algo:ConstructMinima}) 
to such a domain $X$. Still other restrictions are possible, for example 
$X = \{\, (a,b) \in \N\times\N \mid a \le b \text{ and } b \le 2a \,\}$
or more complicated cases such as 
$X = \{\, (a,b) \in \N\times\N \mid a \le b \text{ and } b^2 \le 1+10a \,\}$.
This raises the question as to what are ``reasonable'' 
domains $X \subset A \times B$ to which 
Algorithms \ref{algo:BimonotoneEnumeration} 
and \ref{algo:ConstructMinima} can be efficiently applied.

\subsection{Bimonotone domains} \label{sub:BimonotoneDomains}

As usual we assume that $A$ and $B$ are isotonic
to finite intervals or to the natural numbers.
Figure \ref{fig:minima4} shows a domain $X \subset A \times B$
which will turn out to be well suited to bimonotone enumeration.
Graphically speaking, it is bounded by the graphs of two 
non-decreasing functions $\alpha \colon A \to B$ and 
$\beta \colon B \to A$. We will show that this condition 
suffices to adapt our algorithms to work on the domain $X$
rather than the entire product $A \times B$.

\begin{figure}[htbp]
  \includegraphics[height=50mm]{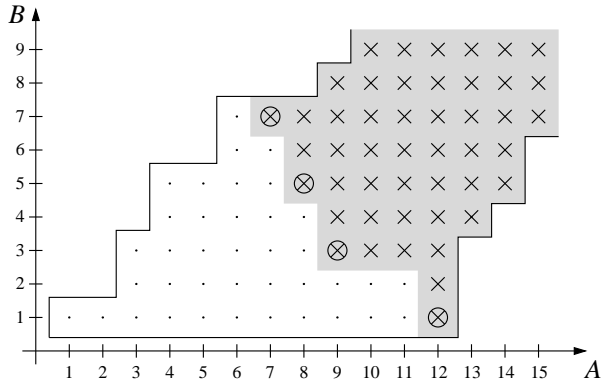}
  \caption{Enumerating a bimonotone domain $X \subset A \times B$}
  \label{fig:minima4}
\end{figure} 

We say that $X \subset A \times B$ is \emph{bounded} 
by functions $\alpha \colon A \to B$ and $\beta \colon B \to A$ if 
\[
X = \{\, (a,b) \in A \times B \mid a \tge \beta(b) \text{ and } b \tge \alpha(a) \,\}.
\]
For example, the domain $X$ of Figure \ref{fig:minima4} 
is bounded by $\alpha(1) = \dots = \alpha(12) = 1$,
$\alpha(13) = 4$, $\alpha(14) = 5$, $\alpha(15) = 7$, 
and $\beta(1) = 1$, $\beta(2) = \beta(3) = 3$, \dots, $\beta(9) = 10$.

\begin{definition}
  We say that a domain $X \subset A \times B$ is \emph{bimonotone}
  if it is bounded by two functions $\alpha \colon A \to B$ and
  $\beta \colon B \to A$ such that:
  \begin{enumerate}
  \item[(1)]
    The functions $\alpha$ and $\beta$ are non-decreasing, that is, \\
    $a \tle a'$ implies $\alpha(a) \tle \alpha(a')$, and 
    $b \tle b'$ implies $\beta(b) \tle \beta(b')$.
  \item[(2)] 
    We have $\beta(\alpha(a)) \tle a$ for all $a \in A$, 
    with equality only for $a = \amin$, \\
    and $\alpha(\beta(b)) \tle b$ for all $b \in B$, 
    with equality only for $b = \bmin$.
  \end{enumerate}
\end{definition}

Condition (2) ensures that $(a,\alpha(a)) \in X$ for each $a \in A$,
and $(\beta(b),b) \in X$ for each $b \in B$. In particular, 
$\alpha$ and $\beta$ are determined by $X$ via
\begin{align*}
  \alpha(a) &= \min \{\, b \in B \mid (a,b) \in X \,\} , \\
  \beta(b)  &= \min \{\, a \in A \mid (a,b) \in X \,\} .
\end{align*}
Moreover, $(\amin,\bmin)$ is the smallest element of $X$.
If both $A$ and $B$ are finite, then $(\amax,\bmax)$ is 
the greatest element of $X$. 

The definition of $X$ via bounding functions is easy 
to formulate and well suited to implementation.
It can also be reformulated in more geometric terms:

\begin{proposition}
  A domain $X \subset A \times B$ is bimonotone if and only if it satisfies 
  $\pr_1 X = A$ and $\pr_2 X = B$ and the following two properties:
  \begin{enumerate}
  \item[(1')]
    If $(a_1,b_1)$ and $(a_2,b_2)$ in $X$ 
    satisfy $a_1 \tle a_2$ and $b_2 \tle b_1$, 
    then $X$ contains the entire rectangle 
    $\{\, (a,b) \in A \times B \mid a_1 \tle a \tle a_2
    \text{ and } b_2 \tle b \tle b_1 \,\}$.
  \item[(2')]
    If $(a_1,b_1)$ and $(a_2,b_2)$ in $X$ 
    satisfy $a_1 \tle a_2$ and $b_1 \tle b_2$, 
    then we can go from $(a_1,b_2)$ to $(a_2,b_2)$ 
    within $X$ by repeatedly incrementing $a$ and $b$.
    \qed
  \end{enumerate}
\end{proposition}

The proof is not difficult and will be omitted.

\subsection{Bimonotone enumeration} \label{sub:BimonotoneDomainEnumeration}

We are now in position to generalize our enumeration algorithm to 
a bimonotone domain.  As before, Algorithm \ref{algo:EnumerateBimonotoneDomain}
processes a bidirectional list $M$ and a priority queue $F$.

\begin{algorithm}[htbp]
  \caption{\quad Sorted enumeration of a bimonotone domain}
  \label{algo:EnumerateBimonotoneDomain}
  
  \begin{specification}
    \REQUIRE  a bimonotone domain $X \subset A \times B$
              and a proper bimonotone function $f \colon X \to Z$ 
    \OUTPUT   an enumeration of $X$, monotone with respect to $f$
  \end{specification}
  
  \smallskip\hrule
  
  \begin{algorithmic}[1]
    \STATE Initialize $M \gets \List{ (\amin,\bmin) }$
           and $F \gets \Queue{ f(\amin,\bmin) }$.
    \WHILE{$F$ is non-empty}
      \STATE Remove a minimal element $f(a,b)$ from $F$ and output $(a,b)$.
      \IF{ $(a,b)$ is the last element of the list $M$ }
        \IFTHEN{ $(\succ{a},b) \in X$ }{ insert $(\succ{a},b)$ into $M$ and $f(\succ{a},b)$ into $F$ }
      \ELSE
        \STATE Let $(a^*,b_*)$ be the successor of $(a,b)$ in the list $M$.
        \IFTHEN{ $\succ{a}<a^*$ }{ insert $(\succ{a},b)$ into $M$ and $f(\succ{a},b)$ into $F$ }
      \ENDIF
      \IF{ $(a,b)$ is the first element of the list $M$ }
        \IFTHEN{ $(a,\succ{b}) \in X$ }{ insert $(a,\succ{b})$ into $M$ and $f(a,\succ{b})$ into $F$ }
      \ELSE
        \STATE Let $(a_*,b^*)$ be the predecessor of $(a,b)$ in the list $M$.
        \IFTHEN{ $\succ{b}<b^*$ }{ insert $(a,\succ{b})$ into $M$ and $f(a,\succ{b})$ into $F$ }
      \ENDIF
      \STATE Remove $(a,b)$ from the list $M$.
    \ENDWHILE
  \end{algorithmic}
\end{algorithm}

\begin{proposition}
  Algorithm \ref{algo:EnumerateBimonotoneDomain} is correct.
\end{proposition}

\begin{proof}
  The proof is essentially the same as 
  for Algorithm \ref{algo:BimonotoneEnumeration}.
  There are, however, some modifications when
  updating the list $M$ and the priority queue $F$:
  \begin{itemize}
  \item
    If the current minimum $(a,b)$ is somewhere in the middle 
    of the list $M$, then the previous arguments apply without change, 
    because we still have
    \[
    \{ (a,b) \}^\bupper \minus \{ (a,b) \} = \{\, (a,\succ{b}), (\succ{a},b) \,\}^\bupper.
    \]
  \item
    If $(a,b)$ is at the end of the list, then possibly 
    $(\succ{a},b) \notin X$: in this case we have
    $\{ (a,b) \}^\bupper \minus \{ (a,b) \} = \{ (a,\succ{b}) \}^\bupper$,
    so we discard $(\succ{a},b)$. 
  \item
    If $(a,b)$ is at the beginning of the list, then possibly
    $(a,\succ{b}) \notin X$: in this case we have 
    $\{ (a,b) \}^\bupper \minus \{ (a,b) \} = \{ (\succ{a},b) \}^\bupper$,
    so we discard $(a,\succ{b})$.
  \item
    If ever $M = \List{(a,b)}$ and neither $(a,\succ{b})$ nor $(\succ{a},b)$ 
    is in $X$, then $(a,b)$ is the greatest element of $X$ 
    and the algorithm terminates correctly.
  \end{itemize}
  Since $f$ is proper, every element $(a,b) \in X$ will eventually be enumerated.
\end{proof}

\subsection{Initialization on a bimonotone domain} \label{sub:BimonotoneDomainInitialization}

As for the unrestricted case $X = A \times B$, 
we want to formulate an initialization algorithm 
for a proper bimonotone function $f \colon X \to Z$ 
defined on some bimonotone domain $X \subset A \times B$.
The idea is essentially the same: Algorithm \ref{algo:BimonotoneMinima}
traces the contour of $X(z) = \{\, x \in X \mid f(x) \tge z \,\}$ 
to construct the list $M = \Min X(z)$ of its minimal elements.

\begin{algorithm}[htbp]
  \caption{\quad Constructing the set of minima on a bimonotone domain}
  \label{algo:BimonotoneMinima}
  
  \begin{specification}
    \REQUIRE  a bimonotone domain $X \subset A \times B$
              and a proper bimonotone function $f\colon X \to Z$
    \INPUT    a level $z\in Z$        
    \OUTPUT   the list of minima $\Min \{\, x\in X \mid f(x) \tge z \,\} 
              = \List{(a_1,b_1),(a_2,b_2),\dots,(a_m,b_m)}$
    \ENSURE   $a_1 < a_2 < \dots < a_m$ and  $b_1 > b_2 > \dots > b_m$
  \end{specification}

  \smallskip\hrule
  
  \begin{algorithmic}[1]
    
    \STATE Initialize $M \gets \emptyset$ and $a \gets \amin$, $b \gets \bmin$

    \WHILE{ $(a,b) \in X$ and $f(a,b) < z$ }
      \WHILEDO{ $f(a,b) < z$ and $(a,\succ{b}) \in X$ }{ $b \gets \succ{b}$ }
      \IFTHEN{ $f(a,b) < z$ }{ $a \gets \succ{a}$ }
    \ENDWHILE

    \WHILE{ $(a,b) \in X$ }
      \WHILEDO{ $(a,\pred{b}) \in X$ and $f(a,\pred{b}) \tge z$ }{ $b \gets \pred{b}$ }
      \STATE Insert $(a,b)$ at the end of the list $M$; 
             continue with $a \gets \succ{a}$, $b \gets \pred{b}$
      \WHILEDO{ $(a,b) \in X$ and $f(a,b) < z$ }{ $a \gets \succ{a}$ }
    \ENDWHILE

    \RETURN $M$

  \end{algorithmic}
\end{algorithm}

\begin{remark}
  % As in Algorithm \ref{algo:ConstructMinima}, 
  The loops in lines 3, 7, and 9 implement linear searches.
  This can be improved by a binary search whenever
  the application requires such optimization.
\end{remark}

\begin{proposition}
  Algorithm \ref{algo:BimonotoneMinima} is correct.
\end{proposition}

\begin{proof}
  The first loop (lines 1--5) finds an element $(a,b) \in X(z)$ 
  with minimal $a$. % such that $a \in A$ is minimal with this property.
  Beginning with $a \gets \amin$ and $b \gets \bmin$, 
  we repeatedly increment $b$ in order to arrive at $f(a,b) \tge z$.
  If this is not possible within $X$, then the candidate $a$ 
  is eliminated, and we continue with $a \gets \succ{a}$.
  If we never run out of the domain $X$, then we finally
  end up with $f(a,b) \tge z$, because $a$ or $b$ 
  increase and $f$ is proper.
  
  The only obstacle occurs when $f(a,b) < z$ but 
  neither $(a,\succ{b})$ nor $(\succ{a},b)$ are in $X$.
  But in this case we have reached the greatest element of $X$,
  hence $f(x) < z$ for all $x \in X$. Thus $X(z) = \emptyset$ 
  and we correctly return the empty list $M = \emptyset$.
  In any case, the first loop terminates with either 
  $(a,b) \notin X$ or $f(a,b) \tge z$, as desired.
  
  When arriving at line 7 we know that $(a,b) \in X(z) \minus M^\bupper$, and 
  $a$ is minimal with this property. The loop in line 7 minimizes $b$, 
  so we know that $(a,b)$ is a minimal element of $X(z)$.
  We thus add $(a,b)$ to our list $M$ % of minimal elements, % (line 8)
  and continue with $a \gets \succ{a}$ and $b \gets \pred{b}$.
  We then repeatedly increment $a$ in order to arrive at $f(a,b) \tge z$.
  If this is not possible in $X$, then $X(z) = M^\bupper$ by the
  rectangle condition (1'), so we have found all minimal elements of $X(z)$. 
  Otherwise, we obtain $(a,b) \in X(z) \minus M^\bupper$, 
  and $a$ is minimal with this property. We can thus 
  reiterate by looping back to line 7.
  
  During each iteration, $a$ is strictly increasing 
  while $b$ is strictly decreasing. We conclude that
  the second loop terminates and produces the list $M$
  of minima, as desired, ordered in the sense that 
  $a_1 < a_2 < \dots < a_m$ and  $b_1 > b_2 > \dots > b_m$.
\end{proof}

%%%%%%%%%%%%%%%%%%%%%%%%%%%%%%%%%%%%%%%%%%%%%%%%%%%%%%%%%%%%%%%%%%%%%%%%%%%%%

\section{Applications to diophantine enumeration} \label{sec:Applications}

Algorithms \ref{algo:EnumerateBimonotoneDomain} and 
\ref{algo:BimonotoneMinima} for bimonotone enumeration 
have been implemented as a class template in C++.
This seems to be a good compromise between general 
applicability, ease of use, and high performance. 
The source files are available on the author's homepage:

\centerline{\texttt{http://www-fourier.ujf-grenoble.fr/$\sim$eiserm/software}}

As an illustration of sorted enumeration, let us mention searching multiple values 
of a polynomial function $f\colon\N\times\N\to\Z$, %% given by 
$f(a,b) = \sum_{i,j} c_{ij} a^i b^j$ with non-negative coefficients $c_{ij} \in \N$. 
The cited implementation has been successfully tested to reproduce some known results 
taken from Richard Guy's \textit{Unsolved problems in number theory} \cite{Guy:2004}.

\subsection{The quest for the sixth taxicab number} \label{sub:TaxiQuest}

As an illustrative example we briefly sketch the taxicab problem.
The $k$th taxicab number, denoted by $\taxicab{k}$,
is the least positive integer that can be expressed 
as a sum of two positive cubes in $k$ distinct ways, 
up to order of summands. That is, it is the smallest
$k$-fold value of $f(a,b)=a^3+b^3$ defined on
$X = \{\, (a,b) \in \N \times \N \mid 1 \le a \le b \,\}$. 

G.\,H.\,Hardy and E.\,M.\,Wright \cite[Thm.\,412]{HardyWright:1954}
proved that, for every $k \ge 1$, there exist such $k$-fold values.
This guarantees the existence of a least $k$-fold value, 
that is, the $k$th taxicab number.
Unfortunately the construction given in the proof
is of no help in finding the \emph{least} $k$-fold value.
Apart from (variants of) exhaustive search,
no such method is known today.
The first taxicab number is trivially 
\[
\taxicab{1} = 2 = 1^3 + 1^3.
\]
The next taxicab numbers are:
\begin{align*}
  \taxicab{2} 
  &= 1729 = 1^3 + 12^3 = 9^3 + 10^3 ,
\end{align*}
(re)discovered by Ramanujan according to Hardy's famous anecdote, 
but previously published by Bernard Fr\'enicle de Bessy in 1657,
\begin{align*}
  \taxicab{3} 
  &= 87\,539\,319 \\
  &= 167^3 + 436^3
  = 228^3 + 423^3
  = 255^3 + 414^3 ,
\end{align*}
discovered by John Leech \cite{Leech:1957} in 1957,
\begin{align*}
  \taxicab{4} 
  &= 6\,963\,472\,309\,248
  = 2421^3 + 19083^3 
  = 5436^3 + 18948^3 \\
  &= 10200^3 + 18072^3
  = 13322^3 + 16630^3 ,
\end{align*}
discovered by E.\,Rosenstiel, J.A.\,Dardis, and C.R.\,Rosenstiel
\cite{Rosenstiel:1991} in 1991,
\begin{align*}
  \taxicab{5} 
  &= 48\,988\,659\,276\,962\,496
  = 38787^3 + 365757^3
  = 107839^3 + 362753^3 \\
  &= 205292^3 + 342952^3
  = 221424^3 + 336588^3
  = 231518^3 + 331954^3 ,
\end{align*}
discovered independently by D.W.\,Wilson \cite{Wilson:1999} in 1997 
and shortly afterwards by D.J.\,Bernstein \cite{Bernstein:2001} in 1998.
%% (David W. Wilson found the fifth taxicab number on Nov. 21, 1997. 
%% A few months later, Daniel J. Bernstein came upon the same number 
%% in an independent investigation.)
Finally, the smallest known $6$-fold value is 
\begin{align*}
  T &= 24\,153\,319\,581\,254\,312\,065\,344 \\
  &= 28906206^3 + 582162^3
  = 28894803^3 + 3064173^3
  = 28657487^3 + 8519281^3 \\
  &= 27093208^3 + 16218068^3
  = 26590452^3 + 17492496^3
  = 26224366^3 + 18289922^3 ,
\end{align*}
found by R.L.\,Rathbun in 2002. 
Is this actually the sixth taxicab number, 
or is there a smaller solution?

\subsection{Feasibility of an exhaustive search} \label{sub:TaxiFeasability}

% taxicab logfile data :
% level=24153319581254312065344 length=369039037733393 width=5963352 frequency=227758/s

In order to verify that $T$ is indeed the smallest $6$-fold value,
there are exactly $n = 369\,039\,037\,733\,393 < 4 \cdot 10^{14}$ pairs 
$(a,b) \in \N\times\N$ to be checked with $a^3 + b^3 \le T$ and $a\le b$.
Such counting results can easily be obtained from Algorithm \ref{algo:BimonotoneMinima} 
tracing the contour of the set $X = \{ f \le z \}$: as a by-product, 
the initialization can be used to determine the sizes 
$n = \sharp \{ f \le z \}$ and $m = \sharp \Min\{ f > z \}$.

\textit{Memory requirements} are, fortunately, no problem.
In the worst case we would have to check all $n$ parameters, which would 
finally build up a priority queue of size $m = 5\,963\,352 < 6 \cdot 10^6$.
Notice that each entry requires $32$ bytes: $12$ bytes for the value $f(a,b)$,
$4$ bytes for $a$ and $4$ bytes for $b$, plus $4$ bytes for each of the three pointers.
In the worst case the priority queue thus requires $180$ megabytes of memory,
which fits nicely in a PC with $256$ megabytes RAM. 
Such memory requirements seem acceptable; on today's PCs 
such a task can reasonably be run in the background.

\textit{Time requirements}, however, are on the edge of being feasible.
Updating a priority queue of $2\cdot 10^6$ entries, say, takes about $4000$ CPU cycles.
On a PC running at 2GHz, we can expect to process about $500\,000$ 
steps per second, that is around $4 \cdot 10^{10}$ steps per day.
This is not too far away from $4 \cdot 10^{14}$, but on a single computer
the search would still require about $10\,000$ days, roughly $25$ years.
On $25$ computers, however, we would be done within a year, possibly earlier.

% taxicab logfile data :
% 2005/05/25 16:29:56 level=485372831914560208129 length=27277510000000 
%   width=1621275 frequency=222222/s

\textit{Partial results.} Up to June 2005, I have run 
the search on a few available PC's at the Institut Fourier, 
but the use of parallelization has still been rather limited (to a dozen PC's).
As a result I obtained the lower bound $\taxicab{6} > 5 \cdot 10^{20}$
by sorted enumeration of the $2.8 \cdot 10^{13}$ smallest values
of $f(a,b) = a^3 + b^3$. (At a speed of $500\,000$ values per second
this takes about $650$ days on a single computer.)
This leaves us with the inequality
\[
5 \cdot 10^{20} < \taxicab{6} \le T \approx 2.42 \cdot 10^{22}
\]
It will now be a matter of sufficient hardware and patience to find the exact answer.

%%%%%%%%%%%%%%%%%%%%%%%%%%%%%%%%%%%%%%%%%%%%%%%%%%%%%%%%%%%%%%%%%%%%%%%%%%%%% 

\bibliographystyle{amsplain}
\bibliography{bimonotone}

\providecommand{\bysame}{\leavevmode\hbox to3em{\hrulefill}\thinspace}
\providecommand{\MR}{\relax\ifhmode\unskip\space\fi MR }
\providecommand{\MRhref}[2]{%
  \href{http://www.ams.org/mathscinet-getitem?mr=#1}{#2}
}
\providecommand{\href}[2]{#2}
\begin{thebibliography}{10}

\bibitem{Bernstein:2001}
D.~J. Bernstein, \emph{Enumerating solutions to {$p(a)+q(b)=r(c)+s(d)$}}, Math.
  Comp. \textbf{70} (2001), no.~233, 389--394. \MR{2001f:11203}

\bibitem{Ekl:1996}
R.~L. Ekl, \emph{Equal sums of four seventh powers}, Math. Comp. \textbf{65}
  (1996), no.~216, 1755--1756. \MR{97a:11050}

\bibitem{GrahamKnuthPatashnik:1989}
R.~L. Graham, D.~E. Knuth, and O.~Patashnik, \emph{Concrete mathematics},
  Addison-Wesley Publishing Co., Reading, Massachusetts, 1989. \MR{MR1001562
  (91f:00001)}

\bibitem{Guy:2004}
R.~K. Guy, \emph{Unsolved problems in number theory}, third ed., Problem Books
  in Mathematics, Springer-Verlag, New York, 2004. \MR{MR2076335}

\bibitem{HardyWright:1954}
G.~H. Hardy and E.~M. Wright, \emph{An introduction to the theory of numbers},
  third ed., The Clarendon Press, Oxford University Press, New York, 1954.
  \MR{MR0067125 (16,673c)}

\bibitem{Knuth:vol1}
D.~E. Knuth, \emph{The art of computer programming, volume 1: fundamental
  algorithms}, second ed., Addison-Wesley Publishing Co., Reading,
  Massachusetts, 1969.

\bibitem{Knuth:vol3}
\bysame, \emph{The art of computer programming, volume 3: sorting and
  searching}, second ed., Addison-Wesley Publishing Co., Reading,
  Massachusetts, 1998.

\bibitem{Leech:1957}
J.~Leech, \emph{Some solutions of {D}iophantine equations}, Proc. Cambridge
  Philos. Soc. \textbf{53} (1957), 778--780. \MR{MR0090602 (19,837f)}

\bibitem{Rosenstiel:1991}
E.~Rosenstiel, J.~A. Dardis, and C.~R. Rosenstiel, \emph{The four least
  solutions in distinct positive integers of the {D}iophantine equation
  {$s=x^3+y^3=z^3+w^3=u^3+v^3=m^3+n^3$}}, Bull. Inst. Math. Appl. \textbf{27}
  (1991), no.~7, 155--157. \MR{MR1125858 (92i:11134)}

\bibitem{Wilson:1999}
D.~W. Wilson, \emph{The fifth taxicab number is
  {$48\,988\,659\,276\,962\,496$}}, J. Integer Seq. \textbf{2} (1999), Article
  99.1.9, HTML document (electronic). \MR{MR1722364 (2000i:11195)}

\end{thebibliography}

%%%%%%%%%%%%%%%%%%%%%%%%%%%%%%%%%%%%%%%%%%%%%%%%%%%%%%%%%%%%%%%%%%%%%%%%%%%%%
\end{document}